\documentclass{article}
\usepackage{graphicx}
\usepackage{amsfonts}
\usepackage{amssymb}
\usepackage{amsmath}
\usepackage{amsthm}
\usepackage{pstricks}
\newcommand{\C}{\mathbb{C}}
\newcommand{\R}{\mathbb{R}}

\newcommand{\Z}{\mathbb{Z}}
\newcommand{\A}{\mathbb{A}}

\newcommand{\Q}{\mathbb{Q}}

\newcommand{\Tr}{{\rm Tr}}
\newcommand{\X}{{\bf X}}
\newcommand{\la}{\langle}

\newcommand{\ra}{\rangle}
\renewcommand{\Im}{{\rm Im}}
\renewcommand{\Re}{{\rm Re}}
\newcommand{\HCd}{{\rm H}_\C^2}
\newcommand{\HCn}{{\rm H}_\C^n}
\newcommand{\HRd}{{\rm H}_\R^2}
\newcommand{\HRn}{{\rm H}_\R^n}

\newcommand{\Pf}{{\em Proof}. }
\newcommand{\EPf}{\hfill$\Box$\vspace{.5cm}}
\numberwithin{equation}{section}

\newtheorem{thm}{Theorem}[section]
\newtheorem*{thm*}{Theorem}
\newtheorem{cor}[thm]{Corollary}
\newtheorem{lem}{Lemma}[section]
\newtheorem{prop}{Proposition}[section]
\newtheorem{dfn}{Definition}[section]

\theoremstyle{remark}
\newtheorem{rmk}{Remark}

\title{Real reflections, commutators and cross-ratios in complex hyperbolic space}
\author{Julien Paupert \footnote{Author partially supported by NSF grant DMS 1007340 and SNF grant 200021-131967/1.}\\
SoMSS \\
Arizona State University\\
P.O. Box 871804\\
Tempe, AZ 85287-1804, USA\\
e-mail: {\tt paupert@asu.edu}
\and 
Pierre Will\\
Institut Fourier\\
Universit\'e de Grenoble I\\
100 rue des Maths\\
38042 St-Martin d'H\`eres, France\\
e-mail: {\tt pierre.will@ujf-grenoble.fr}}

\begin{document}
\maketitle

\begin{abstract} We provide a concrete criterion to determine whether or not two given elements of PU(2,1)
 can be written as products of real reflections, with one reflection in common. As an application, we show that 
the Picard modular groups ${\rm PU}(2,1,\mathcal{O}_d)$ with $d=1,2,3,7,11$ are generated by real
  reflections up to index 1, 2, 4 or 8.
\end{abstract}

\section{Introduction}
A classical way of studying isometries of a symmetric space $\mathcal{S}$ is to decompose
them as products of involutions. The most comfortable situation is to have a class 
$\mathcal{C}$ of isometric involutions with the following properties.

\begin{enumerate}
 \item Any two involutions in $\mathcal{C}$ are conjugate in Isom($\mathcal{S}$).
 \item Any element of Isom($\mathcal{S}$) can be written as a product $s_1s_2$, with $s_i\in\mathcal C$.
\end{enumerate}

For example, if such a family of involutions exists, describing the fixed points of an isometry $A=s_1s_2$ of 
$\mathcal{S}$ amounts to studying the relative position of the fixed point sets of $s_1$ and $s_2$ 
which are totally geodesic subspaces and are isometric to one another. If such a family $\mathcal{C}$ exists, one 
usually says that Isom($\mathcal{S}$) has involution length 2 with respect to $\mathcal{C}$. Of course this requirement is too 
optimistic in general (for example, it fails in Euclidean space of dimension at least 3).

Assuming that a symmetric space has this property, the next question is to decide when two isometries $A$ and $B$ can 
be decomposed using a common involution. We will call such a pair \textit{decomposable} with respect to 
$\mathcal{C}$ (see Definition \ref{defdecomp}; the term \textit{linked} is also commonly used, see \cite{BM}). This  property simplifies the 
understanding of the group $\la A,B\ra$ as it reduces to studying the relative position of 3 pairwise 
isometric totally geodesic subspaces.


One of the most elementary cases is that of the Poincar\'e disk $\Delta$. There are two classes of 
involutions in ${\rm Isom}(\Delta)$, namely \textit{half-turns} and \textit{reflections}. Viewing $\Delta$ as 
the unit disk in $\C$, half-turns are conjugate in ${\rm Isom}^+(\Delta)$ to $z\longmapsto -z$ and 
reflections to $z\longmapsto \overline z$. It is a classical fact that ${\rm Isom}^+(\Delta)$ has 
involution length 2 with respect to reflections, and 3 with respect to half-turns.  
Moreover, any pair of orientation-preserving isometries of the Poincar\'e disk can be decomposed in the form 
\begin{equation}\label{codec}A=s_1s_2 \mbox{ and }B=s_2s_3,\end{equation}
 with the $s_i$ either all half-turns or all reflections. This makes the description of many properties of the group $\la A,B\ra$ 
easier. For instance, when the $s_i$ are all orientation-reversing, the group $\la A,B\ra$ has index two in 
$\Gamma=\la s_1,s_2,s_3 \ra$. In particular, $\la A,B \ra$ is discrete if and only if $\Gamma$ is. The decomposition property we are interested in shares many features with this example, as our involutions are antiholomorphic.

In \cite{BM}, Basmajian and Maskit have studied this question for any space form, that is for Euclidean space, real hyperbolic 
space and the sphere. In particular, they prove that any orientation-preserving isometry is a product of two involutions. They also 
show for instance that pairs in ${\rm Isom}^+(\HRn)$ are generically non decomposable when $n\geqslant 4$.

In this paper we study the question of decomposability in the complex hyperbolic plane, 
which can be seen via a projective model as the unit ball in $\C^2$ equipped with a PU(2,1)-invariant metric. In fact, 
PU(2,1) is the group of holomorphic isometries of $\HCd$; it is the identity component of Isom($\HCd$), the other 
connected component consisting of all antiholomorphic isometries. In particular, PU(2,1) has index two in Isom($\HCd$) and 
in this context holomorphicity  plays the role of preservation of orientation for space forms. The class of involutions we are interested in consists of antiholomorphic involutions, which are all conjugate in PU(2,1) to the map given in affine 
ball coordinates by
\begin{equation}\sigma_0 : (z_1,z_2) \longmapsto(\overline{z_1},\overline{z_2}).\end{equation}
Clearly, $\sigma_0$ fixes pointwise the set of real points of $\HCd$, and conjugates of $\sigma_0$ by 
elements of PU(2,1) fix pointwise \textit{real planes}, which are (totally real) totally geodesic embedded 
copies of the Poincar\'e disk. We will refer to these antiholomorphic involutions as 
\textit{real reflections}.

It has been known since Falbel and Zocca (see \cite{FZ}) that the involution length of PU(2,1) with respect to 
real reflections is 2, and this fact has been generalized to all dimensions by Choi \cite{Cho}. The main question 
we address in this paper is the following. \\

\textit{When is a pair of elements of ${\rm PU}(2,1)$ decomposable with respect to real reflections?}\\

We will abbreviate this by saying that a pair $(A,B)\in {\rm PU}(2,1)^2$ is $\textit{$\R$-decomposable}$ (see Definition~\ref{defdecomp}). 
There is no hope that generic pairs of elements in PU(2,1) are 
$\R$-decomposable. A rough argument for this is the following dimension count. The group PU(2,1) has dimension 8, thus
the product PU(2,1)$\times$PU(2,1) has dimension 16. On the other hand, the set of real planes in $\HCd$ has 
dimension $5$ (to see this note that the stabilizer of the set of real points of the ball is PO(2,1) which has 
dimension $3$). As a real reflection is determined by its fixed real plane, the set of triples of real 
reflections has dimension $15$ (and therefore cannot be diffeomorphic to ${\rm PU}(2,1)\times {\rm PU}(2,1)$). In fact this count leads us to expect that $\R$-decomposability of a pair of isometries is determined by a single (codimension 1) condition, and this will turn out to be the case.

This question has been examined in \cite{W2}, where it was proved that, under the assumption that 
$A$ and $B$ are loxodromic, the pair $(A,B)$ is $\R$-decomposable provided that the trace of the 
commutator $[A,B]$ is real. However, this result was obtained as a byproduct of a classification of pairs of 
elements of PU(2,1) by traces; namely the data $(\Tr A,\Tr B,\Tr AB, \Tr A^{-1}B)$ determines the pair 
$(A,B)$ up to PU(2,1)-conjugation (modulo an order two indetermination, corresponding to the sign of the 
imaginary part of $\Tr[A,B]$). The present approach is more natural, and in particular it allows us to remove the assumption that $A$ and $B$ are loxodromic.

Our main result (Theorem \ref{decomp}) is the following.
\begin{thm*}
Let $A,B \in {\rm PU}(2,1)$ be two isometries not fixing a common point in $\overline{\HCd}$. Then: the pair $(A,B)$ is $\R$-decomposable 
if and only if the commutator $[A,B]$ has a fixed point in $\overline{\HCd}$ whose associated eigenvalue is real and positive.
\end{thm*}
Note that the eigenvalues of $[A,B]$ do not depend on the choice of lifts of $A$ and $B$ to U(2,1). We now sketch our strategy . Any fixed point $p_1$ of $[A,B]$ in $\overline{\HCd}$ gives rise to a cycle of points
\begin{equation}p_1\overset{B^{-1}}{\longrightarrow}p_2\overset{A^{-1}}{\longrightarrow}
p_3\overset{B}{\longrightarrow}p_4\overset{A}{\longrightarrow}p_1.\label{intro-4-cycle}\end{equation}
A key fact is that a real reflection $\sigma$ decomposes $A$ and $B$ if and only if it satisfies $\sigma(p_1)=p_3$ and 
$\sigma(p_2)=p_4$ (Lemma~\ref{exchange}). 
Thus we see that the decomposability of $(A,B)$ is equivalent to a specific symmetry of the above 4-cycle. Now, the existence of 
this symmetry can be detected via cross-ratios. Goldman has proved that an ideal tetrahedron has this 
symmetry if and only if the cross-ratio $\X(p_2,p_4,p_1,p_3)$ is real and positive (Lemma 7.2.1 of \cite{G}). We extend this fact to quadruples of points in 
$\overline{\HCn}$ (Proposition \ref{Xreal}). The last ingredient is to connect the cross-ratio of the 4-cycle \eqref{intro-4-cycle} to the eigenvalue of 
$[A,B]$ associated to the fixed point $p_1$. Denoting this eigenvalue by $\lambda_1$, we will see that the product $\lambda_1\cdot\X(p_2,p_4,p_1,p_3)$ 
is real and positive for any pair $(A,B)$ (this is relation \ref{linkcycleeigenval}). This means that $\lambda_1$ is real if and 
only if $X(p_2,p_4,p_1,p_3)$ is, and both quantities have the same sign, which gives the result. For the sake of completeness we 
then analyze what happens in the two special cases where $\lambda_1$ is negative (section \ref{section-maximal}), and when 
$A$ and $B$ have a common fixed point (section \ref{section-fix-commun}).
\begin{itemize}
 \item When $\lambda_1$ is negative, the fixed point of $[A,B]$ is always on the boundary of $\HCd$ and the group $\la A,B\ra$ preserves a 
complex line. This is an example of a maximal representation of the fundamental group of the once punctured torus, in the sense of 
\cite{BIW}.
\item To describe the situation when $A$ and $B$ have a common fixed point, the main ingredient is a detailed description of 
conditions guaranteeing that a real reflection decomposes a given isometry (this is Proposition \ref{decompCrefllox}). Most cases already appear in the literature, except when $A$ and $B$ are 
parabolic with a common fixed point. In particular, when $A$ and $B$ are both 3-step unipotent 
(see the definitions in section~\ref{generalites}), we use the relative position of their \textit{invariant fans}, as 
defined by Goldman and Parker  in \cite{GP}.
\end{itemize}

A classical and difficult question in complex hyperbolic geometry is to determine the discreteness or non-discreteness of a 
given finitely generated subgroup $\Gamma$ of PU(2,1), and to obtain a presentation of $\Gamma$. Even when discreteness is known 
from general results (for instance in the case of arithmetic lattices), finding a presentation of the 
group is a difficult problem. In fact, the most frequent method is to construct a fundamental domain and use the Poincar\'e 
Polyhedron theorem. This is a very technical task which requires a detailed understanding of the action of $\Gamma$ on $\HCd$.

There are not so many examples of explicit discrete subgroups of PU(2,1), and most of them are obtained from groups with 2 generators.
In this case, the existence of a decomposition as in \eqref{codec} connects an algebraic property 
of the group (being an index 2 subgroup of a group generated by 3 involutions) to a geometric property 
(the existence of totally geodesic fixed point sets in a certain configuration). In other words, $\Gamma$ appears as a 
\textit{reflection group}. Such decompositions appeared for instance naturally in \cite{FalPar2} or \cite{W3} which studied 
certain representations of Fuchsian groups in PU(2,1). They were also central in the constructions of fundamental domains for 
Mostow's lattices (\cite{M}) in \cite{DFP}, as well as for the new non-arithmetic lattices obtained in \cite{ParPau}, \cite{Pau}, 
\cite{DPP1} and \cite{DPP2}. 

In all these occurrences, the existence of the real reflections decomposing the generators required some work 
(part of the detailed geometric construction of the fundamental domains), whereas the concrete criterion given by Theorem 
\ref{decomp} allows us to easily reprove that the Mostow and Deraux-Parker-Paupert lattices are generated by real reflections. 
More importantly, the careful analysis of decomposability conditions for parabolic isometries in Proposition~\ref{decompCrefllox}~(c) 
allows us to show that new lattices are generated by real reflections, namely certain Picard modular groups.

The Picard modular groups $\Gamma_d={\rm SU}(2,1,\mathcal{O}_d)$ (the subgroup of ${\rm SU}(2,1)$ consisting of matrices with 
entries in $\mathcal{O}_d$, where $d$ is a positive squarefree integer and $\mathcal{O}_d$ denotes the ring of
integers of $\Q [\sqrt{-d}]$) are the simplest kind of arithmetic lattice in ${\rm SU}(2,1)$. However their explicit 
algebraic or geometric properties such as generators, fundamental domains and presentations are still unknown in all but very few cases. (Much is known however concerning their volumes, see \cite{H}, \cite{St} and more generally Prasad's volume 
formula \cite{Pra}). More specifically, presentations and fundamental domains have been obtained
for ${\rm PU}(2,1,\mathcal{O}_d)$ when $d=3$ (the so-called Eisenstein-Picard modular group) in \cite{FalPar} and when $d=1$ (the
Gauss-Picard modular group) in \cite{FFP}. More recently, generators for ${\rm PU}(2,1,\mathcal{O}_d)$ with 
$d=2,7,11$ were given in \cite{Z} (these values of $d$ are the ones for which the
ring $\mathcal{O}_d$ is Euclidean). Note that, for $d \neq 3$, ${\rm SU}(2,1,\mathcal{O}_d) \simeq {\rm PU}(2,1,\mathcal{O}_d)$ as 
there are no non-trivial cube roots of unity in $\mathcal{O}_d$. 

The Picard modular groups ${\rm PU}(2,1,\mathcal{O}_d)$ are  analogous to the Bianchi groups ${\rm PGL}(2,\mathcal{O}_d)$ in 
${\rm PGL}(2,\C)$. Bianchi proved in the seminal paper \cite{Bi} that the Bianchi groups are \emph{reflective}, i.e. generated by 
reflections up to finite index, for $d \leqslant 19$, $d \neq 14,17$. At the end of the 1980's, Shaiheev extended these results in 
\cite{Sh}, using results of Vinberg, proving in particular that only finitely many of the Bianchi groups are reflective. (The finiteness result now follows from Agol's result, \cite{Ag}). The classification of 
reflective Bianchi groups was recently completed in \cite{BeMc}.

We prove that  the Picard modular groups ${\rm PU}(2,1,\mathcal{O}_d)$ with $d=1,2,3,7,11$ are reflective; more precisely that they 
have a subgroup of index 1,2 or 4 which is generated by real reflections (Corollary~\ref{picardrefl}).

The paper is organized as follows. We start with some geometric preliminaries in section 2, then study configurations of points and 
cross-ratios in section 3. Section 4 contains the statement and proofs of our main results, which we then apply to various discrete 
subgroups of PU(2,1) in section 5.

\section{Geometric preliminaries}

\subsection{Complex hyperbolic space and isometries\label{generalites}}
The standard reference for complex hyperbolic geometry is \cite{G}. For the reader's convenience we include a brief summary of key 
definitions and facts. Our main result concerns the case of dimension $n=2$, but the general setup is identical for higher dimensions so we state it for all $n \geqslant 1$. 

\paragraph{Distance function:}
Consider $\C^{n,1}$, the vector space $\C^{n+1}$ endowed with a Hermitian form $\langle \cdot \, , \cdot \rangle$ of signature $(n,1)$.
Let $V^-=\left\lbrace Z \in \C^{n,1} | \langle Z , Z \rangle <0 \right\rbrace$.
Let $\pi: \C^{n+1}-\{0\} \longrightarrow \C{\rm P}^n$ denote projectivization.
Define ${\rm H}_\C^n$ to be $\pi(V^-) \subset \C{\rm P}^n$, endowed with the distance $d$ (Bergman metric) given by:

\begin{equation}\label{dist}
  \cosh ^2 \frac{1}{2}d(\pi(X),\pi(Y)) = \frac{|\langle X, Y \rangle|^2}{\langle X, X \rangle  \langle Y, Y \rangle}
\end{equation}


 Different choices of Hermitian forms of signature $(n,1)$ give rise to different models of $\HCn$. The two most standard choices are 
the following. First, when the Hermitian form is given by $\la Z,Z\ra=|z_1|^2+\cdots +|z_n|^2-|z_{n+1}|^2$, the image of $V^-$ under 
projectivization is the unit ball of $\C^n$, seen in the affine chart $\{ z_{n+1}=1\}$ of $\C P^n$. This model is referred to as the \textit{ball model} of $\HCn$. 
Secondly, when $\la Z,Z\ra=2{\rm Re} (z_1\overline{z_{n+1}}) +|z_2|^2+\cdots +|z_n|^2$, we obtain the so-called \textit{Siegel model} of $\HCn$, 
which generalizes the Poincar\'e upper half-plane. More details on the Siegel model in dimension 2 will be given in the next section.

\paragraph{Isometries:}
From \eqref{dist} it is clear that ${\rm PU}(n,1)$ acts by isometries
on ${\rm H}_\C^n$, where ${\rm U}(n,1)$ denotes the subgroup of ${\rm
  GL}(n+1,\C)$ preserving $\langle \cdot , \cdot \rangle$, and ${\rm
  PU}(n,1)$ its image in ${\rm PGL}(n+1,\C)$. In fact, PU($n$,1) is the group of holomorphic isometries of 
${\rm H}_\C^n$, and the full group of isometries is ${\rm PU}(n,1) \ltimes \Z/2$, where the $\Z/2$ factor 
corresponds to a real reflection (see below). Holomorphic isometries of $\HCn$ can be of three types, depending on the number 
and location of their fixed points. Namely, $g \in {\rm PU}(n,1)$ is :
\begin{itemize}
\item \emph{elliptic} if it has a fixed point in ${\rm H}_\C^n$
\item \emph{parabolic} if it has (no fixed point in ${\rm H}_\C^n$ and)
  exactly one fixed point in $\partial{\rm H}_\C^n$
\item \emph{loxodromic}: if it has (no fixed point in ${\rm H}_\C^n$ and) exactly two fixed points in $\partial{\rm H}_\C^n$
 \end{itemize}

\paragraph{Totally geodesic subspaces and related isometries:} 
 A {\it complex k-plane} is a projective $k$-dimensional subspace of 
$\C P^n$ intersecting $\pi(V^-)$
 non-trivially (so, it is an isometrically embedded copy of ${\rm
   H}_\C^{k} \subset {\rm H}_\C^n$). Complex 1-planes are usually
 called {\it complex lines}. If $L=\pi(\tilde{L})$ is a complex $(n-1)$-plane, any
 $v \in \C^{n+1}-\{ 0\}$ orthogonal to $\tilde{L}$ is called a {\it
   polar vector} of $L$. Such a vector satisfies $\langle v,v \rangle
 >0$, and we will usually normalize $v$ so that $\langle v,v \rangle
 =1$.

A {\it real k-plane} is the projective image of a totally real
 $(k+1)$-subspace $W$ of $\C^{n,1}$, i. e. a $(k+1)$-dimensional real
 vector subspace such that $\langle v,w \rangle \in \R$ for all $v,w
 \in W$.  We will usually call real 2-planes simply real planes, or
 $\R$-planes. Every real $n$-plane in ${\rm H}_\C^n$ is the
 fixed-point set of an antiholomorphic isometry of order 2 called a {\it real
   reflection} or $\R$-reflection. The prototype of such an isometry
 is the map given in affine coordinates by $(z_1,...,z_n) \mapsto
 (\overline{z_1},...,\overline{z_n})$ (this is an isometry provided
 that the Hermitian form has real coefficients).

 An elliptic isometry $g$ is called {\it regular} if any of its matrix
 representatives $A \in {\rm U}(n,1)$ has distinct eigenvalues. The
 eigenvalues of a matrix $A \in U(n,1)$ representing an elliptic
 isometry $g$ have modulus one. Exactly one of these eigenvalues has
 eigenvectors in $V^-$ (projecting to a fixed point of $g$ in
 ${\rm H}_\C^n$), and such an eigenvalue will be called {\it of negative
   type}. Regular elliptic isometries have an isolated fixed point in
 ${\rm H}_\C^n$. A non regular elliptic isometry is called {\it special}. 
 Among the special elliptic isometries are the following two types (which exhaust all special elliptic types when $n=2$):

\begin{enumerate}
 \item A {\it complex reflection} is an elliptic
 isometry $g\in {\rm PU}(n,1)$ whose fixed-point set is a complex
 $(n-1)$-plane. In other words, any lift such an isometry to U(n,1) has $n$ 
equal eigenvalues, one of which has negative type.
 
  \item A {\it complex reflection in a point} is an elliptic isometry having a lift with $n$ 
equal eigenvalues, the remaining one being of negative type. In other words, such an isometry 
is conjugate to some $\lambda{\rm Id} \in U(n)$, where $U(n)$ is the stabilizer of the origin 
in the ball model. Complex reflections in a point with order 2 are also called {\it central involutions}; these are the 
symmetries that give ${\rm H}_\C^n$ the structure of a symmetric space.
\end{enumerate}


 A parabolic isometry is called {\it unipotent} if it has a unipotent lift in ${\rm
   U}(n,1)$. If not, it is called {\it screw-parabolic}, and it can be uniquely decomposed as 
 $g=pe=ep$ with $p$ unipotent and $e$ elliptic (see
 Theorem~\ref{conjclasses} below).  In dimensions $n>1$, unipotent 
 isometries are either  {\it 2-step} or {\it 3-step}, according to whether the minimal polynomial of their unipotent lift
 is $(X-1)^2$ or $(X-1)^3$ (see section 3.4 of \cite{ChGr}). 
\subsection{Models in dimension 2}
\subsubsection{The ball model of $\HCd$} The ball model of $\HCd$ arises from the choice of Hermitian form
$$ 
H=
\begin{bmatrix}
1 & 0 & 0 \\
0 & 1 & 0 \\
0 & 0 & -1 
\end{bmatrix}.
$$
It is classical, and we refer the reader to chapter 3 of \cite{G}. We only emphasize the following fact: any elliptic isometry of 
$\HCd$ is conjugate to one given in ball coordinates by $(z_1,z_2) \mapsto (e^{i\alpha}z_1,e^{i\beta}z_2)$ for some $\alpha,\beta \in \R/\Z$. A matrix representative in SU(2,1) for the latter is:

\begin{equation}\label{classellip}
 E_{ (\alpha,\beta)}=\begin{bmatrix}e^{i(2\alpha-\beta)/3} & 0 & 0 \\0 & e^{i(2\beta-\alpha)/3} &0\\0 & 0 &  e^{-i(\alpha+\beta)/3} 
\end{bmatrix}.
\end{equation}

\subsubsection{The Siegel model of $\HCd$}
In the presence of parabolic elements, it is very convenient to use the Siegel domain, as 
the stabilizer of the point at infinity (see below) consists of upper triangle matrices. As mentioned in the previous section, this 
model corresponds to the Hermitian form given by the matrix:
$$ 
H=
\begin{bmatrix}
0 & 0 & 1 \\
0 & 1 & 0 \\
1 & 0 & 0 
\end{bmatrix}
$$ 
In this model, any point  $m \in \HCd$ admits a unique lift to $\C^{2,1}$ of the following form, called its \emph{standard lift}:
\begin{equation}
 {\bf m}=\begin{bmatrix}
          (-|z|^2-u+it)/2\\z\\1
         \end{bmatrix}\mbox{ with } (z,t,u)\in\C\times\R\times]0,\infty[ .
\end{equation}

\noindent The triple $(z,t,u)$ is called the \textit{horospherical coordinates} of $m$. The boundary of $\HCd$ is 
the level set $\lbrace u=0\rbrace$, together with the distinguished point at infinity, given by
$$
q_\infty\sim\left[
\begin{array}{c}
1 \\ 0 \\ 0
\end{array}\right].
$$
 Level sets $\lbrace u=u_0\rbrace$ with fixed $u_0>0$ are called \emph{horospheres based at $q_\infty$}. 
The boundary $\partial\HCd\setminus \{ q_\infty \}$ is a copy of the Heisenberg group $\mathfrak{N}$ of 
dimension 3, with group law given in $[z,t]$ coordinates by:
\begin{equation}
\label{Heisprod}[z_1,t_1]\cdot [z_2,t_2]=[z_1+z_2,t_1+t_2+2{\rm Im}(z_1\overline{z_2})].
\end{equation}
The stabilizer of $q_\infty$ in ${\rm SU}(2,1)$ consists of upper triangular matrices, and is generated by the following 3 types 
of isometries: Heisenberg translations $T_{[z,t]}$ ($(z,t)\in \R\times \C$), Heisenberg rotations $R_\theta$ ($\theta \in \R/2\pi\Z$) 
and Heisenberg dilations $D_r$ ($r >0$), where:
\begin{equation}\label{stabinf}
\begin{array}{ccc}
T_{[z,t]}=
\begin{bmatrix}
1 & -\overline{z} & -(|z|^2-it)/2 \\
0 & 1 & z \\
0 & 0 & 1
\end{bmatrix}
&
R_\theta=
\begin{bmatrix}
e^{-i\theta/3} & 0 & 0 \\
0 & e^{2i\theta/3} & 0 \\
0 & 0 & e^{-i\theta/3}
\end{bmatrix}
&
D_r=
\begin{bmatrix}
r & 0 & 0 \\
0 & 1 & 0 \\
0 & 0 & 1/r
\end{bmatrix}.
\end{array}
\end{equation}
 In Heisenberg coordinates, they correspond  respectively to the following transformations
\begin{itemize}
 \item $T_{[z,t]}$ is the left multiplication by $[z,t]$,
 \item $R_\theta$ is given by $[w,s]\longmapsto [e^{i\theta}w,s]$,
 \item $D_r$ is the Heisenberg dilation $[w,s]\longmapsto [rw,r^2s]$.
\end{itemize}
Note that Heisenberg translations and rotations preserve each horosphere based at $q_\infty$ whereas
Heisenberg dilations permute horospheres based at $q_\infty$. We will denote by ${\rm Isom}(\mathfrak{N})$ the 
non-loxodromic stabilizer of $q_\infty$ in ${\rm SU}(2,1)$, which is generated by the $T_{[z,t]}$ and $R_\theta$. 
The notation ${\rm Isom}(\mathfrak{N})$ comes from the fact it is the isometry group of the \textit{Cygan metric}, 
which we will not use here (see \cite{FalPar}). The group ${\rm Isom}(\mathfrak{N})$ consists exactly of those 
matrices of the form:
\begin{equation}
P_{(z,t,\theta)}=T_{[z,t]}R_{\theta}=\begin{bmatrix}
e^{-i\theta/3} & -e^{2i\theta/3}\overline{z} & -e^{-i\theta/3}(|z|^2-it)/2 \\
0 & e^{2i\theta/3} & e^{-i\theta/3}z \\
0 & 0 & e^{-i\theta/3}
\end{bmatrix}
\end{equation}
It is sometimes more convenient to work with the lift of $P_{(z,t,\theta)}$ to U(2,1) given by 
\begin{equation}\label{PztU21}
\begin{bmatrix}
1 & -\overline{z}e^{i\theta} & -(|z|^2-it)/2\\
0 & e^{i\theta} & z\\
0 & 0 & 1       
\end{bmatrix} 
\end{equation}

In terms of these parameters the various parabolic conjugacy classes are easily described. Any parabolic isometry  is conjugate in 
PU(2,1) to exactly one of the following:
\begin{itemize}
 \item $P_{(1,0,0)}=T_{[1,0]}$ if it is 3-step unipotent,
 \item $P_{(0,1,0)}$ if it is 2-step unipotent,
 \item $P_{(0,1,\theta)}$ for some non-zero $\theta\in\R/2\pi\Z$ if it is screw parabolic. 
\end{itemize}
\begin{rmk}
Screw and 2-step unipotent parabolics share the property of preserving a (unique) complex line. In the case of $P_{(0,1,\theta)}$, it is the line polar to the vector $\begin{bmatrix}0 & 1 & 0\end{bmatrix}^T$. As we will see in the next section, the situation is slightly more complicated for 3-step unipotent parabolics.
\end{rmk}
Finally, recall from \cite{FalPar} for future reference that the exact sequence:
\begin{equation}\label{Pi} 1 \longrightarrow \R \longrightarrow \mathfrak{N} \overset{\Pi}{\longrightarrow} \C \rightarrow 1
\end{equation}
induces an exact sequence:
\begin{equation}\label{exact}
 1 \longrightarrow{\rm Isom}^+ (\R) \longrightarrow {\rm Isom} (\mathfrak{N}) \overset{\Pi_*}{\longrightarrow} {\rm Isom}^+ (\C) \rightarrow 1.
\end{equation}
Explicitly: $\Pi_*(P_{(z,t,\theta)})=\left(\begin{matrix} e^{i\theta} & z \\ 0 & 1\end{matrix}\right)$, acting on $\C$ by 
$w \mapsto e^{i\theta}w+z$.

\subsection{The invariant fan of a 3-step unipotent parabolic\label{section-invariant-objects}}
The Siegel model is very well adapted to describing the action of parabolic isometries. To do so, we give a 
few more details on the structure of the boundary of $\HCd$. It is equipped with a (CR) contact 
structure, which is given in Heisenberg coordinates as the kernel of the 1-form
$$\alpha=dt-2xdy+2ydx.$$
We will denote by $C_m$ the contact plane at a point $m\in\partial\HCd$. 
For any real plane $R$ of $\HCd$, the intersection $\overline R\cap\partial\HCd$ is a closed curve, which is called an 
$\R$-circle (see \cite{G}). These curves are Legendrian: they are everywhere tangent to the contact distribution. The following lemma concerns those $\R$-circles containing the point $q_\infty$, called \emph{infinite $\R$-circles}. We refer to chapter 
4 of \cite{G} for proofs.
\begin{lem}\label{lem-Rplane-infinite}
\begin{enumerate}
\item Let $R$ be a real plane containing the point $q_\infty$ of the Siegel model. Then in Heisenberg coordinates, the $\R$-circle 
$\partial R$ is an affine line in the Heisenberg group, which is contained in the contact plane at any of its points, and the 
restriction to the boundary of the reflection about $R$ is the (Euclidean) half-turn about this affine line. 
\item More precisely, an affine line in the Heisenberg group is an infinite $\R$-circle if and only if it is contained in the
contact plane at one of its points.
\end{enumerate}
\end{lem}
Viewing the boundary of $\HCd$ as the one point compactification of the Heisenberg group 
$\mathfrak{N}=\C\ltimes\R$ endows the set of complex lines through $q_\infty$ with the stucture of the affine space 
$\C$. Indeed, for any $z\in\C$ there exists a unique complex line through $q_\infty$ which contains $q_\infty$ and 
the point $[z,0]$ (which is polar to the vector $\begin{bmatrix} \overline z & 1 & 0 \end{bmatrix}^T$). This 
induces a projection $\tilde{\Pi}:\overline{\HCd}\setminus q_\infty\longmapsto\C$ whose fibers are the complex lines 
through $q_\infty$. In restriction to the boundary, this projection is just the vertical projection $\Pi$ defined in Equation~\eqref{Pi}, which is given in Heisenberg coordinates by:
\begin{equation}\Pi: \,[z,t]\longmapsto [z,0].\label{eqprojboundary}\end{equation}

A \textit{fan} through $q_\infty$ is the preimage of any affine line in $\C$ under the projection 
$\tilde{\Pi}$. In view of \eqref{eqprojboundary}, if $L$ is an affine line in $\C$, the fan $\tilde{\Pi}^{-1}(L)$ intersects the 
boundary of $\HCd$ along the vertical plane containing $L$ in the Heisenberg group. A general fan is the image of 
a fan through $q_\infty$ by an element of PU(2,1). These objects were defined by Goldman and Parker in \cite{GoP} 
(see also chapter 4 of \cite{G}). As stated in \cite{GoP}, fans enjoy a double foliation, by real planes and 
complex lines. We now make this foliation explicit in the case of a fan $F$ through $q_\infty$ in $\HCd$, with $\partial F$ projecting to an affine line $L \subset \C \subset \partial \HCd$:
\begin{enumerate}
 \item First, the foliation by complex lines is given by the fibers of $\tilde{\Pi}$ above the affine line $L$. In the boundary, this 
foliation correspond to the foliation of the vertical plane above $L$ by vertical lines.
\item Consider a point $m$ in the vertical plane $\partial F$. The contact plane at $m$ intersects $\partial F$ along an 
affine line $L'=C_m\cap\partial F$ in the Heisenberg group, which projects vertically onto $L$. By Lemma \ref{lem-Rplane-infinite}, the line 
$L'$ is the boundary of a real plane. Then $\partial F$ is foliated by the family of lines 
parallel to $L'$ contained in $\partial F$. In other words, the foliation of $\partial F$ is obtained by taking all lifts 
of $L$ to the contact structure. All these lines are boundaries of real planes, and this foliation of $\partial F$
 extends inside $\HCd$ as a foliation of $F$ by real planes.
\end{enumerate}
We can be a bit more explicit.
\begin{lem}\label{lem-explicit-boundary-foliation}
Let $L_{w,k}$ be the affine line in $\C$ parametrized by $L_{w,k}=\{w(s+ik),s\in\R\}$, for some unit modulus 
$w$ and $k\geq 0$. Then the boundary foliation of the fan above $L_{w,k}$ is given by the lines parametrized 
in Heisenberg coordinates by $L_{t_0}=\{[w(s+ik),t_0+2sk], s\in\R\}$.
\end{lem}
\begin{proof}
The lines $L_{t_0}$ all project onto $L$ by the vertical projection. A tangent vector to $L_{t_0}$ is given by 
$(\Re(w),\Im(w),2k)$. Evaluating the 1-form $\alpha$ on this vector at the point of parameter $s=1$, and using 
$|w|=1$ shows that the line $L_{t_0}$ is in the kernel of $\alpha$ at this point, thus in the contact plane. 
Therefore $L_{t_0}$ is an $\R$-circle by Lemma \ref{lem-Rplane-infinite}.
\end{proof}
Our interest in fans comes from the following fact.
\begin{prop}\label{prop-invariant-fan}
Let $P$ be a 3-step unipotent parabolic isometry of $\HCd$. There exists a unique fan $F_{P}$ through the fixed point of $P$ such 
that 
\begin{enumerate}
 \item The fan $F_P$ is stable under $P$.
 \item Every leaf of the foliation of $F_P$ by real planes is stable under $P$.
\end{enumerate}
Moreover, when $P$ fixes $q_\infty$ and thus corresponds to left-translation by $[z,t]$ with $z \neq 0$ in Heisenberg coordinates,
the fan $F_P$ is the one above the affine line $L_{w,k}$, where $w=z/|z|$, and $k=t/(4|z|)$.
\end{prop}
Given a 3-step unipotent parabolic isometry $P$, we will refer to the fan $F_P$ as \textit{the invariant fan of $P$},  even though all the fans corresponding to distinct parallel vertical planes are stable under $P$ (these are not stable leaf by leaf).
\begin{proof}
First assume that $P=T_{[1,0]}$, acting on the Heisenberg group as
$$[z,t]\longmapsto [z+1,t-2\Im(z)].$$
Then every vertical plane $\Im(z)=k$ is globally preserved. The real foliation of the fan corresponding 
to the vertical plane $\Im(z)=k$ is given by the family of lines $L_{t_0}=\{[s+ik,t_0+2sk],s\in\R\}$. Since
$$[1,0]\cdot[s+ik,t_0+2sk]=[s+1+ik,t_0+2(s-1)k],$$
we see that the real foliation of the vertical plane $\Im(z)=k$ is preserved by $P$ if and only if $k=0$, which gives the 
result for this normalization of $P$. The first part of the proposition is then obtained by using the fact that any 3-step unipotent 
parabolic is conjugate in PU(2,1) to $T_{[1,0]}$. To check the last part, write $w=z/|z|$ and $k=t/(4|z|)$; then by a direct 
calculation
\begin{eqnarray*}
[z,t]\cdot[w(s+ik,t_0+2sk]&=&[w(|z|+s+ik),t_0+t+2sk-2|z|k]\\
    &=&[w(|z|+s+ik),t_0+2(|z|+s)k],
\end{eqnarray*}
which proves that the real leaves of the fan above $L_{w,k}$ are preserved.
\end{proof}
\noindent In fact the proof of Proposition \ref{prop-invariant-fan} gives us a little more information, which we summarize 
in the following corollary.
\begin{cor}\label{charach}
Let $P$ be a 3-step unipotent parabolic.
\begin{enumerate} 
 \item A real plane is stable under $P$ if and only if it is a leaf of the real foliation of its invariant fan.
 \item $P$ is characterized by its invariant fan $F$, and its restriction to $F$.
\end{enumerate}
\end{cor}
\noindent For future reference, let us state the following proposition, which is just gathering together what we just exposed.
\begin{prop}\label{prop-invariant-parab}
Let $P$ be a parabolic isometry in ${\rm PU}(2,1)$.
\begin{enumerate}
 \item If $P$ is screw parabolic or 2-step unipotent, it has a unique invariant complex line.
 \item If $P$ is 3-step unipotent parabolic, then there exists a unique fan $\mathcal{F}$ centered at the 
fixed point of $P$ which is stable by $P$ and such that every leaf of the foliation of $\mathcal{F}$ by 
real planes is stable by $P$.
\end{enumerate}
\end{prop}
\noindent We will also use the following characterization of commuting parabolic isometries:
\begin{lem}\label{lem-parabinf-commute}
Let $P_1$ and $P_2$ be two parabolic isometries fixing $q_\infty$. Then $P_1$ and $P_2$ commute if and only 
if one of the following possibilities occurs.
\begin{enumerate}
 \item Both $P_1$ and $P_2$ are either 2-step unipotent or screw parabolics with the same stable complex line.
 \item Both $P_1$ and $P_2$ are 3-step unipotent with the same fixed point and their invariant fans intersect $\partial\HCd$
along parallel vertical planes.
\end{enumerate}
\end{lem}
\begin{proof}
We only prove the second part, the first one being classical. Two 3-step unipotent maps fixing $q_\infty$ are 
respectively conjugate to the left translations by $[z_1,t_1]$ and $[z_2,t_2]$. Using \eqref{Heisprod}, we see that 
these two translations commute if and only if $z_2\overline{z_1}\in\R$, which is equivalent to saying that their 
invariant fans are parallel.
\end{proof}

\subsection{Eigenvalues and traces}
The following classification of conjugacy classes in ${\rm U}(n,1)$ is
due to Chen--Greenberg (Theorem~3.4.1 of \cite{ChGr}, where the real
and quaternionic cases are treated as well):
\begin{thm}[Chen--Greenberg]\label{conjclasses}  
\begin{itemize}
\item[(a)] Any elliptic element is semisimple, with eigenvalues of norm
  1. Two elliptic elements are conjugate in ${\rm U}(n,1)$ if and only
  if they have the same eigenvalues and the same eigenvalue of
  negative type.
\item[(b)] Any loxodromic element is semisimple, with exactly $n-1$
  eigenvalues of norm 1. Two loxodromic elements are conjugate in
  ${\rm U}(n,1)$ if and only if they have same eigenvalues.
\item[(c)] Any parabolic element is not semisimple, and all its
  eigenvalues have norm 1. It has a unique decomposition $g=pe=ep$
  with $p$ strictly parabolic and $e$ elliptic. Two parabolic elements
  are conjugate in ${\rm U}(n,1)$ if and only if their strictly
  parabolic and elliptic components are conjugate.
\item[(d)] There are 2 classes of strictly parabolic elements for
  $n>1$, the vertical Heisenberg translations and the non-vertical
  Heisenberg translations.
\end{itemize}
\end{thm}
Note that two elliptic elements may be conjugate in ${\rm GL}(n+1,\C)$
but not in ${\rm U}(n,1)$ (if they have the same eigenvalues but
different eigenvalues of negative type). It will be useful for future reference to have matrix representatives in SU(2,1) of the different conjugacy classes of isometries when $n=2$.

\begin{enumerate}
 \item Any elliptic isometry of $\HCd$ is conjugate to one given in the ball model by the matrix $E_{(\alpha,\beta )}$ given in \eqref{classellip}.
 \item Any loxodromic isometry of $\HCd$ is conjugate to one given in the Siegel model by the matrix  $D_r R_\theta$ where $D_r$ and 
$R_\theta$ are as in \eqref{stabinf}.
\item Any parabolic isometry of $\HCd$ is conjugate to one given in the Siegel model by the matrix $P_{(z,t,\theta)}=T_{[z,t]} R_\theta$.
The isometry associated to $P_{(z,t,\theta)}$ is unipotent if and only if $\theta=0$ and 2-step unipotent if and only if $\theta=0$ and $z=0$.
\end{enumerate}

As in the classical case of the Poincar\'e disc, the isometry type of an isometry is closely related 
to the trace of a lift to SU(2,1). The characteristic polynomial of a matrix $A$ in SU(2,1) is given by:
\begin{equation}\label{charpoly}
\chi_A(X)=X^3-z\cdot x^2+\overline{z}\cdot x -1,\mbox{ where $z={\rm Tr}A$}.
\end{equation}
Computing its discriminant, we obtain:
\begin{equation}\label{goldpoly}
f(z)= Res( \chi_A,\chi_A')=\vert z \vert ^4 - 8Re(z^3)+18\vert z \vert ^2-27.
\end{equation}
This function provides  the following classification of holomorphic isometries via the trace
of their lifts to SU(2,1) (see ch. 6 of \cite{G}), which is analogous to the classical SL(2,$\C$) case. 
Denote by $C_3$ the set of cube roots of unity in $\C$.
\begin{thm}[Goldman]
Let $A \in {\rm SU}(2,1)$ and $g \in {\rm PU}(2,1)$ the corresponding
isometry of $H_\C^2$. Then:
\begin{itemize}
\item $g$ is regular elliptic $\iff$ $f({\rm Tr}(A))<0$.
\item $g$ is loxodromic $\iff$ $f({\rm Tr}(A))>0$.
\item $g$ is special elliptic or screw-parabolic $\iff$ $f({\rm Tr}(A))=0$
  and ${\rm Tr}(A) \notin 3C_3$.
\item $g$ is unipotent or the identity $\iff$ ${\rm Tr}(A) \in 3C_3$.  
\end{itemize}
\end{thm}  
The null-locus of the polynomial $f$ can be seen in Figure~\ref{deltoid}.
\begin{figure}
\centering
\scalebox{0.3}{\includegraphics{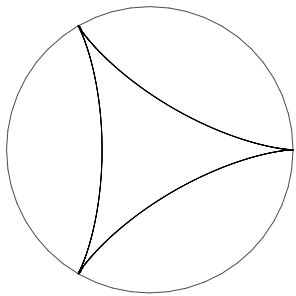}}
\caption{The null-locus of the polynomial $f$ inscribed in the circle
  of radius 3 centered at the origin}\label{deltoid} 
\end{figure}

We now focus on the special case of elements of SU(2,1) having real trace.
\begin{prop}\label{eigenvals} Let $A \in {\rm SU}(2,1)$ satisfy ${\rm
    Tr}A \in \R$. Then $A$ has an eigenvalue equal to 1. More
  precisely:
  \begin{itemize}
  \item If $A$ is loxodromic then $A$ has eigenvalues $\{1,r,1/r \}$
    for some $r \in \R \setminus [-1,1]$.
  \item If $A$ is elliptic then $A$ has eigenvalues
    $\{1,e^{i\theta},e^{-i\theta} \}$ for some $\theta \in [0,\pi]$.
  \item If $A$ is parabolic then $A$ has eigenvalues $\{1,1,1 \}$ or
    $\{1,-1,-1 \}$
  \end{itemize}
\end{prop}

\Pf If ${\rm Tr}A$ is real, then $\chi_A$ has real coefficients, therefore 
the eigenvalue spectrum of $A$ is stable under complex conjugation. The result 
follows. \EPf

\begin{rmk}\label{remvalpone}
\begin{enumerate}

 \item[(a)] Any loxodromic or parabolic element of SU(2,1) with real trace is conjugate to its inverse; for elliptic isometries this is true under the additional assumption that the eigenvalue 1 has negative type. This follows from Proposition~\ref{conjclasses}. 
An element of a group which is conjugate to its inverse is sometimes called \emph{achiral} or \emph{reversible} in the context of isometry groups. In \cite{GonP}, Gongopadhyay and Parker have studied and classified these isometries in ${\rm PU}(n,1)$ for all $n \geqslant 1$.

 \item[(b)] For later use, let us note that matrix representatives of the conjugacy classes of elements of PU(2,1) having a fixed point in 
$\overline{\HCd}$ with associated positive eigenvalue are given by, in the notation of~\eqref{classellip} and \eqref{stabinf}:  
\begin{itemize}
 \item $E_{(\theta,-\theta)}$ with $\theta\in\R$ (in the ball model) for elliptic classes,
 \item $D_r$ with $r>1$ (in the Siegel model) for loxodromic classes, and
 \item $P_{(z,t,0)}$ with $z\in\C$ and $t\in\R$ (in the Siegel model) for parabolic classes.
\end{itemize}
\end{enumerate}
\end{rmk}

\subsection{Antiholomorphic isometries\label{antiholoisom}}
The following lemma is useful when computing with antiholomorphic isometries (see \cite{FalPau} in the elliptic case, where the corresponding matrix was called a \emph{Souriau matrix} for $f$). To simplify the statement, we consider a projective 
model (for example the ball or Siegel model) for which complex conjugation in affine coordinates $\sigma_0:\ (z_1,...,z_n) \mapsto (\overline{z_1},...,\overline{z_n})$ 
is an isometry (i.e. the Hermitian form has real coefficients).

\begin{lem}\label{antihololift}
 Let $f$ be an antiholomorphic isometry of $\HCn$. Then there exists  $M\in {\rm U}(n,1)$ such that for any $m \in \HCn$ 
with lift ${\bf m} \in \C^{n+1}$: 

\begin{equation}
 f(m)=M\cdot\overline{\bf m}.
\end{equation}
Any matrix $M\in {\rm U}(n,1)$ with this property will be called a {\it lift} of $f$.
\end{lem}

\Pf $f\circ\sigma_0$ is a holomorphic isometry, i.e. corresponds to an element of PU($n$,1). Any of its lifts $M \in {\rm U}(n,1)$  
satisfies the required property. \EPf

\begin{dfn}\label{defdecomp} 
\begin{itemize}
\item Given a real reflection $\sigma$ and an isometry $A \in {\rm PU}(n,1)$, we say that $\sigma$ \emph{decomposes} $A$ if 
$A=\sigma \tau$ for some real reflection $\tau$ (equivalently, $A=\tau ' \sigma$ where $\tau ' = \sigma\tau\sigma$ is also 
a real reflection).
\item Given two isometries $A,B\in {\rm PU}(n,1)$, we say that the pair $(A,B)$ is \emph{$\R$-decomposable} if there exists a real 
reflection which decomposes both $A$ and $B$. 
\end{itemize}
\end{dfn}
Note that when writing the two isometries $A$ and $B$ as products of real reflections, the order in which the reflections appear 
is not important. As an application of Lemma \ref{antihololift}, we obtain a necessary condition on a pair $(A,B)$ to be $\R$-decomposable, namely that the 
commutator $[A,B]$ must have real trace. Note that the trace of an element of PU($n$,1) is not well-defined in general. However, if 
$A$ and $B$ are in PU($n$,1), then the matrix $[\tilde A,\tilde B]$ does not depend on the choice of lifts $\tilde A$ and $\tilde B$ 
made for $A$ and $B$. This allows us to consider the condition of having real trace for a commutator.

\begin{lem}\label{lemRdecprodanti}
 If a pair $(A,B)\in {\rm PU}(n,1)\times {\rm PU}(n,1)$ is $\R$-decomposable, then the commutator $[A,B]$ is the square of an anti-holomorphic isometry.
\end{lem}
\begin{proof}
If $A=\sigma_1\sigma_2$ and $B=\sigma_2\sigma_3$ with each $\sigma_i$ a real reflection, then         $[A,B]= (\sigma_1\sigma_2)(\sigma_2\sigma_3)(\sigma_2\sigma_1)(\sigma_3\sigma_2)=   (\sigma_1\sigma_3\sigma_2)^2$, where $\sigma_1\sigma_3\sigma_2$ is an antiholomorphic isometry. \end{proof}

\begin{cor}\label{cororealtrace}
 If a pair $(A,B)\in {\rm PU}(n,1) \times {\rm PU}(n,1)$ is $\R$-decomposable, then: ${\rm Tr}([A,B]) \in \R$.
\end{cor}
\begin{proof}
 First note that if $f_1$ and $f_2$ are antiholomorphic isometries with lifts $M_1$ and $M_2$, 
then $f_1\circ f_2$ is holomorphic with lift $M_1\overline{M_2} \in {\rm U}(n,1)$. Using this fact together with Lemma 
\ref{lemRdecprodanti}, we see that if $(A,B)$ is $\R$-decomposable, then the commutator can be lifted to SU($n$,1) as
$M\overline M$, with $M=M_1\overline{M_3}M_2$, where $M_k$ is a lift of $\sigma_k$. As a consequence, we see that 
$$\overline{\Tr(M\overline{M})}=\Tr(\overline{M}M)=\Tr(M\overline{M})\in\R$$ \end{proof}
Note that this necessary condition holds in any dimension $n$. However, to obtain a sufficient condition, we will need
Lemma 2.4 below, which is false in dimensions $n>2$.

\begin{lem}\label{refl1}
  Any antiholomorphic isometry of ${\rm H}_\C^1$ having a fixed
  point in ${\rm H}_\C^1$ is a real reflection.
\end{lem}

\Pf Let $f$ be an antiholomorphic isometry of ${\rm H}_\C^1$ with a fixed point $p$. Then, for any point $q\neq p$, 
the angles $(f^{-1}(q),p,q)$ and $(q,p,f(q))$ are opposite (because $f$ is antiholomorphic), therefore 
$f^{-1}(q)=f(q)$ (because $f$ is an isometry), and $f$ is an involution and fixes pointwise a geodesic. \EPf

\begin{lem}\label{refl2} If $g$ is an antiholomorphic isometry of ${\rm
    H}_\C^2$ which exchanges two points of $\overline{{\rm H}_\C^2}$,
  then $g$ is an $\R$-reflection.
\end{lem}

\Pf Let $p$ and $q$ be the two points exchanged by $g$. Then 
$g$ has a fixed point $m$ on the geodesic $(pq)$ (which is the
midpoint of the geodesic segment $[pq]$ when $p$ and $q$ are in
$\HCd$). Then $g$ stabilizes the complex line $C$ spanned by $p$ and $q$, as well as the complex 
line $C'$ orthogonal to $C$ at $m$. By Lemma~\ref{refl1}, the restrictions of $g$ to $C$ and $C'$ are real reflections, 
fixing geodesics $\gamma$ and $\gamma'$. Now $\gamma$ and $\gamma'$ are geodesics 
contained in orthogonal complex lines, therefore they span a Lagrangian plane, which is 
fixed pointwise by $g$. \EPf

Lemma \ref{refl2} is false as soon as $n\geq 3$. Indeed, consider an element $M\in$ SU(n,1) given in the Siegel model by
\begin{equation}
 M=\begin{bmatrix}
    (0) &  & z\\
     & U & \\
    1/\bar z &  & (0)
   \end{bmatrix},
\end{equation}
where $z\in\C^\star$ and $U\in U(n-1)$ is a matrix such that $U\overline U\neq Id$, that is a non-symmetric matrix in $U(n-1)$ 
(such matrices only exist when $n>2$). 
Then
$$
M\overline{M}=\begin{bmatrix}
    1 &  & (0)\\
    & U\overline U & \\
    (0) & \dots & 1
   \end{bmatrix}\neq Id
$$
Therefore the antiholomorphic isometry assciated with $M$ is not a real reflection, though it exchanges the two points 
corresponding to the first and last vectors in the canonical basis of $\C^{n+1}$, which are both on the boundary of $\HCn$.


The following proposition describes which elements of PU(2,1) can arise as squares of antiholomorphic isometries. 
\begin{prop}
 Let $A\in {\rm PU}(2,1)$ be an isometry admitting a lift $\tilde A$ to ${\rm SU}(2,1)$ such that for any fixed point $A$ in $\overline \HCd$, 
the corresponding eigenvalue of $\tilde A$ is positive. Then $A$ is the square of an antiholomorphic isometry unless $A$ is 2-step 
unipotent parabolic. 
\end{prop}

\begin{proof}
Being the square of an antiholomorphic isometry is preserved by conjugation, therefore we only need to prove the result for representatives 
of each conjugacy class. We first provide matrices $M$ such that $M\overline M$ is equal to the lifts listed in Remark \ref{remvalpone}(b) in the elliptic and loxodromic cases. 
\begin{enumerate}
 \item If $A$ is elliptic the following $M$ satisfies the above condition: 
 $$M=\begin{bmatrix}0 & e^{i\theta/2} & 0 \\ e^{-i\theta/2} & 0 & 0\\ 0 & 0 & 1\end{bmatrix}$$
 \item If $A$ is loxodromic then we can take:
  $$M=\begin{bmatrix}\sqrt{r} & 0 & 0 \\ 0 & 1 & 0\\ 0 & 0 & 1/\sqrt{r}\end{bmatrix}.$$
 \end{enumerate}
We now examine the case where $A$ is parabolic, thus unipotent in view of the assumptions. Assume that $A$ is unipotent and that $A=\phi^2$ with $\phi$ an antiholomorphic isometry. First, $\phi$ has at least one fixed 
point in the closure of $\HCd$ and $\phi^2$ has exactly one fixed point there. This implies that $\phi$ has only one fixed point, 
which is the same as that of $A$. Conjugating by an element of PU(2,1), we can assume that this fixed point is $q_\infty$. Consider 
a lift $M$ of $\phi$ to SU(2,1) in the sense of Lemma \ref{antihololift}. The fact that $M\overline M$ is a unipotent map fixing 
$q_\infty$,  which corresponds to the vector $\begin{bmatrix}1 & 0 & 0 \end{bmatrix}^T$ implies that $M$ is an upper triangular matrix in SU(2,1)
 with unit modulus diagonal entries. Therefore $M$ is of the form $P_{(z,t,\theta)}$. Computing 
$M\overline{M}=P_{(z,t,\theta)}P_{(\bar z,-t,-\theta)}$, we obtain that $A$ has a lift to SU(2,1) of the form:
\begin{equation}\label{antisquareparab}
\begin{bmatrix} 
1 & -\left(\overline{z}+ze^{-i\theta}\right) & -\overline{z}\left(z+\overline z e^{i\theta}\right)\\ 
0 & 1 & \left(z+\overline{z}e^{i\theta}\right)\\
0 & 0 & 1
\end{bmatrix}
\end{equation}
This matrix can only be 3-step unipotent or the identity. This proves the result as there is only one 
PU(2,1)-conjugacy class of 3-step unipotents.
\end{proof}

\subsection{Decomposing isometries\label{section-decompose-one}}
The following proposition summarizes results characterizing which real reflections decompose a 
given holomorphic isometry. Parts (e1-e3) are Proposition~2.4 of \cite{FalPau}, part (l) is  Proposition~4(2) 
of \cite{W2} (and follows from Proposition~3.1 of \cite{FZ}), and part (p) is new:
\begin{prop}\label{decompCrefllox}  Let $A \in {\rm PU}(2,1)$ and $\sigma$ a real reflection with fixed $\R$-plane L.
\begin{itemize}
\item[(e1)] If $A$ is a complex reflection with fixed complex line $C$, then: $\sigma$ decomposes $A$ if and only if $L \cap C$ is a geodesic. 
\item[(e2)] If $A$ is a complex reflection in a point $p_A$, then: $\sigma$ decomposes $A$ if and only if $p_A \in L$.
\item[(e3)] If $A$ is regular elliptic with fixed point $p_A$ and stable complex lines $C_1,C_2$, then: $\sigma$ decomposes $A$ 
if and only if $p_A \in L$ and $L \cap C_i$ is a geodesic for $i=1,2$.
\item[(l)] If $A$ is loxodromic then: $\sigma$ decomposes $A$ if and only if $\sigma$ exchanges the 2 fixed points of $A$.
\item[(p)] If $A$ is parabolic the situation is as follows.
\begin{enumerate}
 \item  If $A$ is screw parabolic or 2-step unipotent with fixed point $p$ and stable complex line $L$ then $\sigma$ decomposes $A$ 
if and only if $\sigma(p)=p$ and $\sigma(L)=L$.
\item If $A$ is 3-step unipotent with fixed point $p$ and invariant fan $F$ (see Proposition \ref{prop-invariant-parab}), then 
$\sigma$ decomposes $A$ if and only if $\sigma(p)=p$, $\sigma(F)=F$, and the restriction of $\sigma$ to $\partial F$ is a half-turn.
\end{enumerate}
\end{itemize}
\end{prop}
\Pf
Parts (e1-e3) are Proposition~2.4 of \cite{FalPau}, part (l) is  Proposition~4(2) 
of \cite{W2} (and follows from Proposition~3.1 of \cite{FZ}). Let us prove part (p).

Denote by $p$ the fixed point of $A$. Assume that $A=\sigma\sigma'$ where $\sigma'$ is another real reflection. 
Because both $\sigma$ and $\sigma'$ are involutions, either $\sigma$ and $\sigma'$ fix $p$, or there exists 
$q\neq p$ in $\partial\HCd$ such that $\sigma$ and $\sigma'$ both swap $p$ and $q$. In this case, $A$ would 
fix two distinct points in $\partial\HCd$, which is not possible for a parabolic. Thus the fixed real plane of 
$\sigma$ contains $p$. Replacing $p$ by $V$ and using the fact that $V$ is  the unique stable complex line or 
invariant fan of $A$, we obtain that $\sigma$ and $\sigma'$ both preserve $V$.

Conversely, let $\sigma$ be such a real reflection. 
\begin{enumerate}
 \item If $A$ is 2-step unipotent or screw parabolic, then $V$ is a complex line. The restriction of $\sigma$ 
to $V$ is an involution fixing a boundary point of $V$. As $V$ is a copy of the Poincar\'e disk, this implies 
that the restriction $\sigma_{\vert V}$ is a symmetry about a geodesic, and therefore $L$ intersects $V$ along a 
geodesic $\gamma$, one of whose endpoints is the fixed point of $A$. Let us call $a$ the other endpoint of 
$\gamma$, and $b=A(a)$. It is a simple exercise to check that $A\circ \sigma$ exchanges $a$ and $b$. 
Therefore $A\circ \sigma$, which is an antiholomorphic isometry, exchanges two points of $\partial\HCd$. 
By Lemma \ref{refl2}, it is a real reflection and therefore $\sigma$ decomposes $A$.
\item If $A$ is 3-step unipotent, we may assume by conjugating that $A=T_{[1,0]}$; then the  boundary of its 
invariant fan $F_A$  is the vertical plane $\{y=0\}$, on which $A$ acts as a horizontal translation by $1$. In view of Lemma 
\ref{lem-sigma-inf-fan}, the fixed $\R$-circle of $\sigma$ is either a horizontal line $\{t=t_0\}$ contained in the plane $\partial F_A$, or an infinite $\R$-circle 
orthogonal to this vertical plane. 

In the first case $\sigma\circ A$ acts on $\partial F_A$ by $(x,t)\longmapsto (x+1,t_0-t)$ which is not an 
involution, and therefore $\sigma$ cannot decompose $A$. In the second case, the restriction of $\sigma$ to $\partial F_A$ is a 
half-turn, and so is the restriction of $\sigma\circ A$. The latter is also the restriction to $\partial F_A$ of a reflection $\sigma'$ about 
a infinite $\R$-circle; then $A=\sigma\circ\sigma'$ by Corollary \ref{charach}, as $A$ and $\sigma\circ\sigma'$ agree on their common invariant fan.
\end{enumerate}
\EPf

The following lemma describes real reflections that preserve a given fan. We state it only for fans through $q_\infty$, as this is 
all we will need. Let $F$ be a fan through $q_\infty$ and $(L_t)_{t\in\R}$ be the foliation of its boundary by infinite 
$\R$-circles. Let $L$ be an infinite $\R$-circle not contained in or parallel to $\partial F$, and $L_{t_0}$ be the unique leaf of $\partial F$ such that 
$L\cap L_{t_0}$ is nonempty. We will say that $L$ and $F$ are orthogonal whenever $L$ and $L_{t_0}$ are orthogonal in the contact 
plane at the point $L\cap L_{t_0}$. It is a direct consequence of Lemma \ref{lem-Rplane-infinite} that the reflection about an 
infinite $\R$-circle preserves any fan which is orthogonal to it.
\begin{lem}\label{lem-sigma-inf-fan}
Let $F$ be a fan through $q_\infty$. Let $\sigma$ be the real reflection about a real plane $R$. Then $\sigma$ preserves
$F$ if and only $R$ contains the point $q_\infty$ on its boundary and one of the following occurs:
\begin{enumerate}
 \item The real plane $R$ is one of the leaves of $F$, in which case $\sigma$ acts on the vertical plane
$\partial F$ as the reflection across the affine line $\partial R$.
\item The real plane $R$ intersects the Heisenberg group along an affine line orthogonal to the vertical plane $\partial F$. In 
that case, it acts on $\partial F$ as a Euclidean half-turn.
\end{enumerate}
\end{lem}
\begin{proof}
First, $\sigma$ fixes the point $q_\infty$. Indeed, as it preserves $F$, $\sigma$ preserves globally the (singular) foliation of $F$ by 
real planes and therefore fixes the intersection of all leaves, which is $q_\infty$. This means that $R$ contains $q_\infty$, and that 
the corresponding $\R$ circle is infinite. Conjugating by an element of ${\rm Isom}(\mathfrak{N})$, we can assume that $\sigma$ is the
real reflection about the real plane $\HRd\cap\HCd$, which acts on the boundary by $\sigma([z,t])=[\overline z,-t]$. A vertical plane 
is stable under $\sigma$ if and only if it is the vertical plane $\{y=0\}$, in which case $R$ is a leaf of the corresponding fan, or a 
plane orthogonal to the $x$-axis, which corresponds to the second case. 
\end{proof}

\section{Configurations of points and cross-ratios}
\subsection{Triples of points}
\begin{dfn}
Given a triple $(p_1,p_2,p_3)$ of distinct points in $\HCn\cup\partial\HCn$ and lifts $P_i \in \C^{n,1}$ of the $p_i$, the ratio
\begin{equation}
T(p_1,p_2,p_3)=\dfrac{\la P_1,P_2\ra\la P_2,P_3\ra\la P_3,P_1\ra}{\la P_1,P_3\ra\la P_3,P_2\ra\la P_2,P_1\ra}
\end{equation}
does not depend on the lifts $P_i$. We will call $T(p_1,p_2,p_3)$ the \textit{triple-ratio} of $(p_1,p_2,p_3)$.\end{dfn}
Note that $T(p_1,p_2,p_3)$ is also well-defined if 2 or more of the points are equal in $\HCn$ (but not in $\partial \HCn$).

Observe that holomorphic isometries (elements of PU($n$,1)) clearly preserve the triple-ratio, whereas for any 
antiholomorphic isometry $g$, we have $T(g(p_1),g(p_2),g(p_3))=\overline{T(p_1,p_2,p_3)}$. The triple-ratio is related to the 
classical Cartan angular invariant $\A$ (see \cite{C}) and Brehm's shape invariant $\sigma$ (see \cite{Br}) for triangles as 
follows.

\begin{itemize}
\item  The Cartan angular invariant of three points $p_1,p_2,p_3 \in \partial \HCn$ is defined as: 
$$\A(p_1,p_2,p_3)=\arg(-\la
  P_1,P_2\ra\la P_2,P_3\ra\la P_3,P_1\ra).$$ 
  It relates to the triple-ratio by
  $$T(p_1,p_2,p_3)=e^{2i\A(p_1,p_2,p_3)}.$$
\item Brehm's shape invariant $\sigma$ of three points in $\HCn$ is related to the normalized triple product

$$\tilde T(p_1,p_2,p_3)=\dfrac{\la P_1,P_2\ra\la P_2,P_3\ra\la P_3,P_1\ra}{\la P_1,P_1\ra\la P_2,P_2\ra\la P_3,P_3\ra}.$$
Namely, $\sigma=-{\rm Re} (\tilde{T})$. Note that:
$$T(p_1,p_2,p_3)=\dfrac{\tilde T(p_1,p_2,p_3)}{\tilde T(p_1,p_3,p_2)}.$$
\end{itemize}
We refer the reader to chapter 7 of \cite{G} for classical properties of the Cartan invariant. Note in particular 
that the Cartan invariant satisfies $\A(p_1,p_2,p_3)\in[-\pi/2,\pi/2]$, and that $\A(p_1,p_2,p_3)=\pm\pi/2$ 
(resp. $\A(p_1,p_2,p_3)=0$) if and only if the three points are contained in a complex line (resp. a real plane). Also, the 
Cartan invariant classifies triples of pairwise distinct points in $\partial \HCn$ up to holomorphic isometries.

The following classification of triples of points in $\HCn$ is due to Brehm (\cite{Br}).

\begin{thm}\label{Brehm}(Brehm) Let $(x_1,x_2,x_3)$, $(y_1,y_2,y_3)$ be two triples
  of points in ${\rm H}_\C^n$. There exists $g \in {\rm Isom}({\rm
    H}_\C^n)$ such that $g(x_i)=y_i$ (for $i=1,2,3$) if and only if:

\begin{itemize} 

\item $d(x_1,x_2)=d(y_1,y_2)$, $d(x_2,x_3)=d(y_2,y_3)$,
  $d(x_1,x_3)=d(y_1,y_3)$ and
  $T(x_1,x_2,x_3)=T(y_1,y_2,y_3)$, in which case $g$
  is holomorphic, or
 
\item $d(x_1,x_2)=d(y_1,y_2)$, $d(x_2,x_3)=d(y_2,y_3)$,
  $d(x_1,x_3)=d(y_1,y_3)$ and
  $T(x_1,x_2,x_3)=\overline{{T}(y_1,y_2,y_3)}$, in which
  case $g$ is antiholomorphic.
\end{itemize}
\end{thm} 

In fact Brehm's formulation is slightly different as he considers
$-{\rm Re} (\tilde{T})$ instead of $\tilde{T}$ (so his statement doesn't
include our 2 cases). This is equivalent because the norm of
$\tilde{T}$ is determined by the 3 side-lengths.

\subsection{The complex cross-ratio}
The following definition is due in this form to Goldman (\cite{G}) (following Koranyi and Reimann (\cite{KR})) in the case of boundary points:
\begin{dfn}
  Let $(p_1,p_2,p_3,p_4)$ be a quadruple of distinct points in
  $\HCn\cup\partial\HCn$. The quantity defined by
\begin{equation}\label{KRdef} {\bf X}(p_1,p_2,p_3,p_4)=\frac{\langle
    P_3,P_1 \rangle \langle P_4,P_2 \rangle}{\langle P_4,P_1 \rangle
    \langle P_3,P_2 \rangle}
\end{equation}
does not depend on the choice of lifts $P_i$ of the $p_i$'s, and is
called the \textit{complex cross-ratio} of $(p_1,p_2,p_3,p_4)$.
\end{dfn}
Note that ${\bf X}(p_1,p_2,p_3,p_4)$ is also well-defined when some of the 4 points coincide, as long as at most 2 of them coincide in $\partial\HCn$.

 The complex cross-ratio of boundary points has been studied in detail in \cite{G} (pp. 224--228), to which we refer the
reader for more details. As for the triple-ratio, it is a direct observation that holomorphic isometries preserve 
${\bf X}$ whereas antiholomorphic ones change it to its complex conjugate. One of our main tools will be finding conditions 
under which such a cross-ratio is real, in the spirit of the following result (Theorem 7.2.1 of \cite{G}):

\begin{thm}[Goldman]\label{gold}
  Let $(p_1,p_2,p_3,p_4)$  be a quadruple of distinct points in $\partial {\rm H}_\C^n$. Then ${\bf
    X}(p_1,p_2,p_3,p_4)$ is real and positive if and only if there
  exists a real reflection $\phi$ such that $\phi : p_1
  \leftrightarrow p_2$ and $p_3 \leftrightarrow p_4$.
\end{thm}

Note that, if there exists such a real reflection $\phi$, then:
\begin{eqnarray}
 \X(p_1,p_2,p_3,p_4)  =  \overline{\X(\phi(p_1),\phi(p_2),\phi(p_3),\phi(p_4))}
 =  \overline{\X(p_2,p_1,p_4,p_3)}
\end{eqnarray}
Going back to the definition of $\X$, it is straightforward that $\X(p_2,p_1,p_4,p_3)=\X(p_1,p_2,p_3,p_4)$, and we see 
that the condition that $\X\in\R$ is indeed necessary. In \cite{G} the assumption that ${\bf X}>0$ is omitted, but it must be added for the following
reason.  $\X$ is related to triple products by: 

\begin{equation}\label{crosstriple}
\X(p_1,p_2,p_3,p_4)=\dfrac{\la P_1,P_2\ra\la P_2,P_3\ra\la P_3,P_1\ra}{\la P_1,P_2\ra\la P_2,P_4\ra\la P_4,P_1\ra}\cdot\dfrac{|\la P_4,P_2\ra|^2}{|\la P_3,P_2\ra|^2}
\end{equation}
Because the Cartan invariant belongs to $[-\pi/2,\pi/2]$, we see that if $\X$ is real and negative 
the two triple products $\la P_1,P_2\ra\la P_2,P_i\ra \la P_i,P_1\ra$ ($i=3,4$) must have arguments
either both equal to $\pi/2$ or both equal to $-\pi/2$. This means
that $p_3$ and $p_4$ belong to the complex line spanned by $p_1$ and
$p_2$, and are on the same side of the geodesic $(p_1p_2)$. See Proposition 2 of
\cite{KR} and property \# 7 on p. 226 of \cite{G}. However, if
$p_1,p_2,p_3,p_4$ are in such a configuration then there cannot exist
a real reflection $\phi$ such that $\phi : p_1 \leftrightarrow p_2$
and $p_3 \leftrightarrow p_4$. Indeed, if a real reflection preserves
a complex line then it acts on it by reflection in a geodesic. 

The following basic observation will allow us to project  orthogonally onto the
complex sides of the quadrilateral $(p_1,p_2,p_3,p_4)$.

\begin{lem}\label{proj} Let $p_1,p_2,p_3,p_4 \in \overline{{\rm
      H}_\C^n}$ with $p_1 \neq p_2$, and let $\pi_{12}$ denote orthogonal projection onto
  the complex line $L_{12}$ spanned by $p_1$ and $p_2$. Then: ${\bf
    X}(p_1,p_2,p_3,p_4)={\bf X}\left(
    p_1,p_2,\pi_{12}(p_3),\pi_{12}(p_4)\right) $.
\end{lem}

\Pf Let $c_{12}$ be a polar vector for $L_{12}$, normalized so that
$\langle c_{12},c_{12} \rangle =1$. Then $\pi_{12}$ is the
projectivization of the linear projection in $\C^{2,1}$ given by:
$\Pi_{12}(z)=z-\langle z,c_{12} \rangle c_{12}$. Then, for any point $p$ in $\HCn\cup\partial\HCn$:

$$\la \Pi_{12}(p),p_i\ra=\la p-\la p,c_{12}\ra c_{12},p_i\ra=\la p,p_i\ra \  (i=1,2).$$
The result follows by substituting $p_3$ and $p_4$ in this expression. 
(Here and whenever it is convenient we will slightly
abuse notation by using the same letter for points in $\C P^n$ and
their lifts to $\C^{n,1}$; we will however insist that lifts of points
inside ${\rm H}_\C^n$ have norm $-1$ and lifts of points outside
$\overline{{\rm H}_\C^n}$ have norm 1.) \EPf

\subsection{Cross-ratios and real reflections}

Recall that the classical cross-ratio of 4 distinct points in 
$\C P^1=\C \cup \{ \infty \}$ is defined by (see for instance \cite{G}):

\begin{equation}\label{usual}
[z_1,z_2;z_3,z_4]=\dfrac{(z_4-z_1)(z_3-z_2)}{(z_4-z_2)(z_3-z_1)}
\end{equation} 

It is invariant under the diagonal action of ${\rm PGL}(2,\C)$, and is real if and only if the 4
points are cocylic or collinear. Moreover, $[z_1,z_2;z_3,z_4]$ is positive if
and only if the pairs $(z_1,z_2)$ and $(z_3,z_4)$ do not separate each
other on the common circle/line. The KR cross-ratio generalizes the
classical cross-ratio in the following sense:

\begin{lem}\label{classic} If $p_1,p_2,p_3,p_4$ lie in a common
  complex line $C\subset \C P^n$, then  

$${\bf X}(p_1,p_2,p_3,p_4)=[\sigma(p_1), \sigma(p_2);p_4,p_3]$$ 
 where $\sigma$ denotes inversion in the boundary circle of $C$.
\end{lem}

\Pf By applying an element of PU($n$,1), we may assume that the
complex line $C$ containing the $p_i$ is the first coordinate axis
of $\C^n$ (seen as an affine chart of $\C P^n$) in the ball model.  Each point of $C$ has
a lift to $\C^{n,1}$ of the form $[z,0,...,0,1]$, and in these coordinates,
$\sigma$ is given by $z\longmapsto 1/\bar z$. We lift each $p_i$ as a
vector $P_i=[z_i,0,...,0,1]^T$ (the standard lift in the ball model), and
compute:

$$ 
{\bf X}(p_1,p_2,p_3,p_4)=
\frac{(z_3\overline{z_1}-1)(z_4\overline{z_2}-1)}{(z_4\overline{z_1}-1)(z_3\overline{z_2}-1)}
=
\frac{(z_3-1/\overline{z_1})(z_4-1/\overline{z_2})}{(z_3-1/\overline{z_2})(z_4-1/\overline{z_1})}=[\sigma(z_1),
  \sigma(z_2);z_4,z_3] 
$$ \EPf

From Lemma \ref{proj} and \ref{classic} and the properties of the classical cross-ratio we obtain the 
following reality condition for ${\bf X}$:

\begin{prop}\label{Xreal} Let $p_1,p_2,p_3,p_4 \in \overline{{\rm
      H}_\C^n}$ with $p_1 \neq p_2$, and let $\pi_{12}$ denote orthogonal projection onto
  the complex line $L_{12}$ spanned by $p_1$ and $p_2$.  Then:
\begin{enumerate}
\item The complex cross-ratio ${\bf X}(p_1,p_2,p_3,p_4)$ is real and
  positive if and only if
\begin{itemize}
\item either the points
  $p_1,p_2,\pi_{12}(p_3),\pi_{12}(p_4)$ are all equidistant from a
  geodesic $\gamma$ in $L_{12}$, with $ p_1,p_2$ on one side of
  $\gamma$ and $\pi_{12}(p_3),\pi_{12}(p_4)$ on the other, 
\item or the $p_i$ are on the boundary of a common complex line and
  $\{p_1,p_2\}$ does not separate $\{p_3,p_4\}$ on this circle.
\end{itemize}
\item ${\bf X}(p_1,p_2,p_3,p_4)$ is real and negative if and only if $p_1,p_2,p_3,p_4$ are on the
  boundary of a common complex line and $\{p_1,p_2\}$ separates
  $\{p_3,p_4\}$ on this circle.
\end{enumerate}
\end{prop}

Note that the second statement with the $p_i$ on the boundary of ${\rm
  H}_\C^n$ is one half of Proposition 2 of \cite{KR} and of property 7
on p. 226 of \cite{G}. However the statement of the other half (so,
our first statement) is different when some of the $p_i$ are in ${\rm
  H}_\C^n$.
 
\Pf Normalize as in Lemma \ref{classic} so that the complex line
$L_{12}$ containing $p_1$ and $p_2$ is the first coordinate axis in
the ball model of $\HCn$. Denote by $z_1,\cdots,z_4$ the respective
coordinates in this unit disk of the points $p_1,p_2,\pi_{12}(p_3)$,
and $\pi_{12}(p_4)$. According to Lemma \ref{classic} and Proposition
\ref{Xreal},

\begin{eqnarray}{\bf X}(p_1,p_2,p_3,p_4) \in \R & \iff & [1/\overline{z_1},1/\overline{z_2};z_4,z_3]\in \R \nonumber\\
  &\iff & 1/\overline{z_1},1/\overline{z_2},z_4,z_3 \mbox{ lie on a
    common circle $C$ in $\C P^1$}\label{circle}
\end{eqnarray}
Note that $z_1,z_2,z_3,z_4$ are in the closed unit disk
of $\C$, so that $1/\overline{z_1}$ and $1/\overline{z_2}$ are outside
the open unit disk. In particular, either $C$ intersects the unit circle in 2
points $p$ and $q$, or $C$ is the unit circle. 

\begin{enumerate}
 \item  In view of \eqref{circle}, if $\X(p_1,p_2,p_3,p_4)<0$ the two pairs $(z_1,z_2)$ and $(1/\bar z_3,1/\bar z_4)$ separate each other 
on $C$. The latter remark tells us that this is only possible when $C$ is the unit circle. Therefore $p_1$, $p_2$, $p_3$ and $p_4$ 
all belong to the boundary of $L_{12}$ and the pairs $(p_1,p_2)$ and $(p_3,p_4)$ separate each other on $C$. 
\item  Assume that $\X(p_1,p_2,p_3,p_4)>0$. 
\begin{enumerate}
 \item If $C$ is the unit circle, then $z_i=1/\bar z_i$ for $i=1,2$ and thus $\X(p_1,p_2,p_3,p_4)=[z_1,z_2,z_4,z_3]$, which 
 is positive if and only if $\lbrace p_1,p_2\rbrace$ do not separate $\lbrace p_3,p_4 \rbrace$ in $C$.
 \item If $C$ intersects the unit circle in 2 points $p$ and $q$, let $\gamma$ denote the geodesic whose endpoints are $p$ and $q$. Then
$\pi_{12}(z_3)$ and $\pi_{12}(z_4)$ are on a hypercycle with endpoints
$p$ and $q$ (the part of $C$ which is inside the unit disk), and
$z_1$, $z_2$ are on the image of this hypercycle by reflection in
$\gamma$ (this is the image of the other half of $C$ by inversion in
the unit circle), see the left part of Figure 2. Therefore $p_1,p_2,\pi_{12}(p_3),\pi_{12}(p_4)$ are all equidistant from $\gamma$, with $ p_1,p_2$ on one side of
  $\gamma$ and $\pi_{12}(p_3),\pi_{12}(p_4)$ on the other. \EPf
\end{enumerate}
\end{enumerate}

\begin{figure}\label{eyes}
\begin{center}
 \scalebox{0.5}{\includegraphics{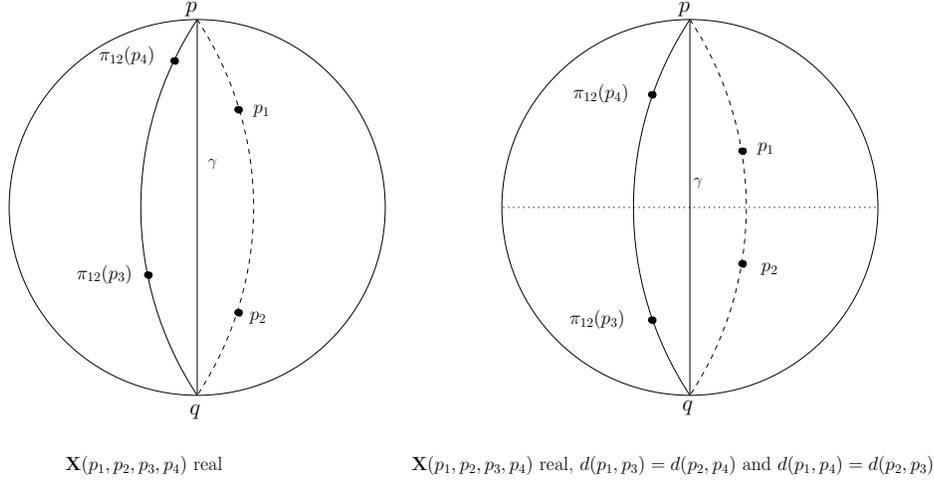}}
\caption{Projections of quadruples $(p_1,p_2,p_3,p_4)$ to $L_{12}$}
\end{center}
\end{figure}

The following result is the analogue of Theorem~\ref{gold} in the case
where the 4 points are inside ${\rm H}_\C^2$:
\begin{thm}\label{moregold} Let $p_1,p_2,p_3,p_4 \in {\rm
    H}_\C^2$ with $p_1 \neq p_2$ or $p_3 \neq p_4$. There exists a real reflection $\phi$ such that $\phi :
  p_1 \leftrightarrow p_2$ and $p_3 \leftrightarrow p_4$ if and only
  if: ${\bf X}(p_1,p_2,p_3,p_4)>0$, $d(p_1,p_3)=d(p_2,p_4)$ and
  $d(p_1,p_4)=d(p_2,p_3)$.
\end{thm}

\Pf Let $p_1,p_2,p_3,p_4
\in {\rm H}_\C^2$ satisfy ${\bf X}(p_1,p_2,p_3,p_4)>0$,
$d(p_1,p_3)=d(p_2,p_4)$ and $d(p_1,p_4)=d(p_2,p_3)$. 
Consider the 2 ordered triples $(p_1,p_2,p_3)$
and $(p_2,p_1,p_4)$; these have equal side-lengths by our
assumptions. We now compare the corresponding triple-ratios,
$T_1=T(p_1,p_2,p_3)$ and $T_2=T(p_2,p_1,p_4)$. Observe that, taking lifts
$P_1,...,P_4$ of $p_1,...,p_4$ in $\C ^{2,1}$:

\begin{eqnarray}
  T_2 & = & \dfrac{\la P_2,P_1\ra\la P_1,P_4\ra\la P_4,P_2\ra}{\la P_1,P_2\ra\la P_4,P_1\ra\la P_2,P_4\ra}\nonumber\\
  & = & \dfrac{\la P_2,P_1\ra\la P_3,P_2\ra\la P_1,P_3\ra}{\la P_1,P_2\ra\la P_2,P_3\ra\la P_3,P_1\ra}\cdot 
  \dfrac{\la P_2,P_3\ra\la P_1,P_4\ra}{\la P_2,P_4\ra\la P_1,P_3\ra}\cdot
  \dfrac{\la P_3,P_1\ra\la P_4,P_2\ra}{\la P_3,P_2\ra\la P_4,P_1\ra}\nonumber\\
  & = & \overline{T_1}\cdot\dfrac{\X(p_1,p_2,p_3,p_4)}{\overline{\X(p_1,p_2,p_3,p_4)}}
\end{eqnarray}
Now by assumption ${\bf X}(p_1,p_2,p_3,p_4)>0$, so that $T_2=\overline{T_1}$.
Then,  by  Theorem \ref{Brehm}, there exists an antiholomorphic isometry $g$
sending $p_1$ to $p_2$, $p_2$ to $p_1$ and $p_3$ to $p_4$. But such an
isometry must be an $\R$-reflection by Lemma \ref{refl2}, assuming that $p_1 \neq p_2$. If $p_1=p_2$ but $p_3 \neq p_4$, the same argument applies, permuting the indices by the permutation $(13)(24)$.  \EPf

It may seem that we only used the assumption that $\X$ was real in the above proof. Recall however that if $\X$ is negative, the four points are
on the boundary of a complex line, and the two pairs $\{p_1,p_2\}$ and
$\{p_3,p_4\}$ separate each other. In that case, as observed above, there cannot exist a real 
reflection $\phi$ such that $\phi(p_1)=p_2$ and $\phi(p_3)=p_4$.

\section{Commutators, decomposable pairs and traces}

\subsection{Main results}

Recall that a pair of holomorphic isometries $(A,B) \in {\rm
  PU}(n,1)^2$ is said to be $\R$-\emph{decomposable} if there exist 3
$\R$-reflections $\sigma_1$, $\sigma_2$ and $\sigma_3$ such that
$A=\sigma_1 \sigma_2$ and $B=\sigma_1 \sigma_3$.  We are now ready to prove our main result.

\begin{thm}\label{decomp}
Let $A,B \in {\rm PU}(2,1)$ be two isometries not fixing a common point in $\overline{\HCd}$. Then: the pair $(A,B)$ is $\R$-decomposable if and only if the commutator $[A,B]$ has a fixed point in $\overline{\HCd}$ whose associated eigenvalue is real and positive.
\end{thm}

Note that the eigenvalues of elements of ${\rm PU}(2,1)$ are not well-defined (up to change of lift in ${\rm U}(2,1)$, or even ${\rm SU}(2,1)$), but the eigenvalues of a commutator are well-defined (the commutator itself is independent of lifts). Using Goldman's classification of isometries by trace
and Proposition~\ref{eigenvals}, this criterion can be reduced to the
following:

\begin{thm}\label{decomp2} Let $A,B \in {\rm PU}(2,1)$ be two isometries not fixing a common point in $\overline{\HCd}$. Then $(A,B)$ is
  $\R$-decomposable if and only if:
\begin{itemize}
\item $[A,B]$ is loxodromic and ${\rm Tr} [A,B]>3$, or

\item $[A,B]$ is unipotent, or
 
\item $[A,B]$ is elliptic, ${\rm Tr} [A,B]\in \R$ and the eigenvalue 1
  of $[A,B]$ is of negative type.
\end{itemize}
\end{thm}

The extra assumption in the elliptic case means that the eigenvalue 1
corresponds to the fixed point of $[A,B]$. The other eigenvalues of
$[A,B]$ are then $e^{\pm i\theta}$ for some $\theta$, by the
assumption that ${\rm Tr} [A,B]\in \R$ and
Proposition~\ref{eigenvals}.

\begin{figure}
 \begin{center}
  \scalebox{0.5}{\includegraphics{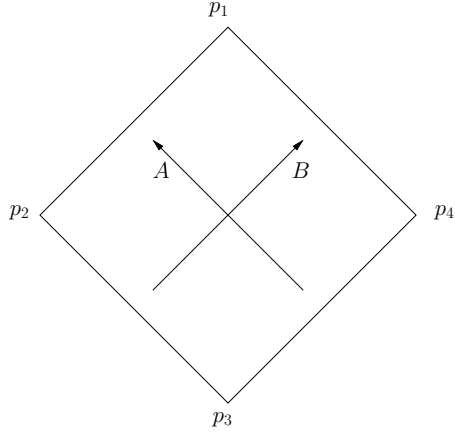}}
\caption{The 4-cycle associated to a fixed point of $[A,B]$\label{4cycle}}
 \end{center}
\end{figure}

\emph{Proof of Theorem~\ref{decomp}:} Let $p_1$ be a fixed point of
$[A,B]$ in $\overline{{\rm H}_\C^2}$, $P_1$ a lift of $p_1$ in
$\C^{2,1}$ and $\lambda_1$ the associated eigenvalue, so that:
$[A,B]P_1=\lambda_1P_1$. Consider the cycle of
four points defined as follows: let $P_2=B^{-1}P_1$, $P_3=A^{-1}P_2$ and
$P_4=B(P_3)$. First assume for simplicity that these 4 points are all distinct. Then opposite sides of the quadrilateral $(P_1P_2P_3P_4)$ are identified
by $A$ and $B$ as on Figure \ref{4cycle}. Note that
$AP_4=[A,B]P_1=\lambda_1P_1$. Then:
\begin{equation}\label{linkcycleeigenval}
{\bf X}(p_2,p_4,p_1,p_3)=\frac{\langle P_1,P_2 \rangle \langle P_3,P_4
  \rangle}{\langle P_3,P_2 \rangle \langle P_1,P_4
  \rangle}=\frac{\langle \lambda_1^{-1}AP_4,AP_3 \rangle \langle
  P_3,P_4 \rangle}{\langle P_3,P_2 \rangle \langle BP_2,BP_3
  \rangle}=\lambda_1^{-1}\frac{|\langle P_3,P_4 \rangle|^2}{|\langle
  P_3,P_2 \rangle|^2}.
\end{equation}
This proves the following:

\begin{lem}\label{positive} The cross ratio ${\bf X}(p_2,p_4,p_1,p_3)$ is real and positive if and only if $\lambda_1 \in \R^+$
\end{lem}

For the next step we use Theorem~\ref{gold} or \ref{moregold},
depending on whether $p_1$ (and hence all other $p_i$'s) is on $\partial
{\rm H}_\C^2$ or in ${\rm H}_\C^2$. In the latter case, by construction
of the 4 points we have $d(p_1,p_2)=d(p_3,p_4)$ and
$d(p_1,p_4)=d(p_2,p_3)$. Therefore Theorem~\ref{moregold} tells us:

\begin{lem}\label{Lem-cross-pos} There exists a real reflection $\phi$ such that 
$\phi (p_1)=p_3$ and $\phi(p_2) = p_4$ if and only if the cross-ratio ${\bf
    X}(p_2,p_4,p_1,p_3)$ is real and positive.
\end{lem}

The following lemma concludes the proof of Theorem~\ref{decomp}:
\begin{lem}\label{exchange} $(A,B)$ is $\R$-decomposable $\iff$ there exists a real
  reflection $\phi$ such that $\phi : p_1 \leftrightarrow p_3$ and
  $p_2 \leftrightarrow p_4$.
\end{lem}
Indeed, if $A=\sigma_2 \sigma_1$ and $B=\sigma_3 \sigma_1$ then
$\sigma_1=\phi$ satisfies $\phi : p_1 \leftrightarrow p_3$ and $p_2
\leftrightarrow p_4$. Conversely, if such a $\phi$ exists then by Lemma~\ref{refl2} above, 
$A\circ \phi$ and $B \circ \phi$ are real reflections. Indeed they are both antiholomorphic, and 
$A\circ \phi$ (resp. $B\circ\phi$) exchanges $p_1$ and $p_2$ (resp. $p_1$ and $p_4$). Therefore $(A,B)$ 
is $\R$-decomposable. \EPf

Finally we examine the case where some of the 4 points $p_1,...,p_4$ are equal. The cross-ratio appearing in Lemma~4.1 is well-defined as long as no three of the points are on $\partial \HCd$ and  equal, in which case all four of them would be equal, contradicting the assumption that $A$ and $B$ do not have a common fixed point. The proofs of Lemmas~4.2 and 4.3 carry through as long as $p_1 \neq p_3$ or $p_2 \neq p_4$. 

Now if $p_1 = p_3=p$ and  $p_2 = p_4=q$, then $A$ and $B$ both exchange $p$ and $q$ (if they are distinct). If $p$ and $q$ are in $\HCd$ then $A$ and $B$ both fix the midpoint of the segment $[pq]$ which is again assumed not to be the case. If $p$ and $q$ are on $\partial \HCd$ then $A$ and $B$ both have a fixed point on the geodesic line $(pq)$ and act as a half-turn on the complex line spanned by $p$ and $q$. In that case, on one hand by Proposition~\ref{decompCrefllox}(a1-a3), any real plane containing $(pq)$ decomposes $A$ and $B$, therefore $(A,B)$ is $\R$-decomposable. On the other hand, the commutator $[A,B]$ is loxodromic (because it acts by translation along the geodeesic $(pq)$) and has an eigenvalue equal to 1 (because $A$ and $B$ have a common eigenvector as they both preserve the complex line spanned by $p,q$), therefore $[A,B]$ has 3 real and positive eigenvalues. In particular, the conclusion of Theorem~\ref{decomp} holds in this case as well. \EPf

\subsection{Groups fixing a point\label{section-fix-commun}}
When $A$ and $B$ have a common fixed point in $\overline{\HCd}$ the results are the following:

\begin{prop}\label{fixedpointinside} If $A,B \in {\rm PU}(2,1)$ have a common fixed point in $\HCd$ then $(A,B)$ is $\R$-decomposable.
\end{prop}
Proposition~\ref{fixedpointinside} is the first part of Theorem~2.1 of \cite{FalPau}; it essentially follows from 
the fact that, given two complex lines in $\C^2$, there exists a Lagrangian subspace intersecting each of them in 
a line (see Proposition~\ref{decompCrefllox} (e1-e3)). 
\begin{prop}\label{fixedpointboundary} Let $A,B \in {\rm PU}(2,1)$ have a common fixed point on $\partial\HCd$. 
\begin{itemize}
\item[(a)] If $A$ or $B$ is loxodromic then $(A,B)$ is $\R$-decomposable if and only if $A$ and $B$ commute.
\item[(b)] If $A$ and $B$ are both non-loxodromic and one of them is not 3-step unipotent then $(A,B)$ is $\R$-decomposable.
\item[(c)] If $A$ and $B$ are both 3-step unipotent then $(A,B)$ is $\R$-decomposable if and only if $A$ and $B$ commute.
\end{itemize}
\end{prop}

Note that the 3 parts of Proposition~\ref{fixedpointboundary} cover all cases where $A$ and $B$ have a common 
fixed point on $\partial\HCd$, because screw-parabolic isometries have 2-step unipotent part (this follows 
from the fact that their elliptic and unipotent parts commute in the classification theorem of Chen-Greenberg, 
our Theorem~2.2(c)).
\begin{proof}[Proof of Proposition~\ref{fixedpointboundary}]
\begin{itemize}
 \item[(a)] First note that if $A$ and $B$ are both loxodromic with a common fixed point, then by Proposition  
\ref{decompCrefllox} (l): $(A,B)$ is $\R$-decomposable if and only if $A$ and $B$ have the same fixed points, that is 
they commute.  

\begin{itemize}
 \item Assume that one of $A$, $B$ is loxodromic but not the other, say $A$ is loxodromic and $B$ is parabolic or a complex reflection, 
and that $A$ and $B$ commute. In particular, $A$ has distinct eigenvalues, therefore $B$ must also be diagonalizable by 
the assumption that $A$ and $B$ commute. Moreover $B$ fixes the fixed points of $A$, which means that the fixed line of $B$ contains 
the axis of $A$. Therefore any real reflection which decomposes $A$ also decomposes $B$  (by Proposition~\ref{decompCrefllox} (e2) 
and (l)).  In particular, $(A,B)$ is $\R$-decomposable.
\item Conversely, assume that $A$ is loxodromic and $B$ is not, and that $(A,B)$ is $\R$-decomposable. By part (p) of 
Proposition~\ref{decompCrefllox}, $B$ cannot be parabolic. It is thus a complex reflection. Denote by $C$ the complex 
line fixed by $B$, and by $p_A,q_A$ the fixed points of A. Then by Proposition~\ref{decompCrefllox} (e2) and (l), there 
exists a real reflection $\sigma$ with fixed $\R$-plane $R$ such that $\sigma$ exchanges $p_A$ and $q_A$, and such that 
$R\cap C$ is a geodesic. Then $\sigma$ preserves $\overline{C}$ which contains one of $p_A$ and $q_A$, therefore $\overline{C}$ 
also contains the other, and $A$ and $B$ commute.
\end{itemize}

\item[(b)] Assume that $A$ is not 3-step unipotent. Conjugating if necessary, we may assume that the common fixed 
point is $q_\infty$. Denote by $C$ the stable complex line of $A$, and pick a point $p$ in the vertical line $\partial C$. Any affine line through $p$ contained in the contact plane at $p$ is the boundary of a real plane $R$ such that $R\cap C$ contains $p$ and 
$q_\infty$. This implies that for any such $R$, $R\cap C$ is the geodesic connecting $p$ to $q_\infty$, and therefore $R$ decomposes $A$ by part (p) of Proposition~\ref{decompCrefllox}.
\begin{itemize}
\item If $B$ preserves a complex line, that is $B$ is a complex reflection or screw parabolic, then its stable 
complex line $C'$ is such that $\partial C'$ is a vertical complex line. If $C=C'$ then the result is clear.
If not, one of the infinite $\R$-circles through $p$ intersects $\partial C'$ and therefore the corresponding real reflection decomposes $B$.
\item If $B$ is 3-step unipotent then one of the infinite $\R$-circles through $p$ intersects the invariant fan of $B$ 
orthogonally, therefore the corresponding real reflection decomposes $B$.
\end{itemize}
\item[(c)] If both $A$ and $B$ are 3-step unipotent then in view of Proposition \ref{decompCrefllox} and Lemma 
\ref{lem-sigma-inf-fan}, the pair $(A,B)$ is $\R$-decomposable if and only if there exists a real plane $R$ such that the infinite 
$\R$-circle $\partial R$ is orthogonal to both invariant fans of $A$ and $B$. This is equivalent to saying the these two fans are parallel, which by Lemma \ref{lem-parabinf-commute} means that $A$ and $B$ commute. 
\end{itemize}
\end{proof}

\subsection{Maximal representations: $\C$-Fuchsian punctured torus groups \label{section-maximal}}
The case where the cross-ratio ${\bf X}(p_2,p_4,p_1,p_3)$ is real and negative corresponds to a rigidity phenomenon, giving the 
following result which holds in all dimensions: 
\begin{prop}\label{extremal} If $[A,B]$ has a fixed point in $\overline{{\rm H}_\C^n}$
  whose associated eigenvalue is real and negative, then the group
  $\langle A, B \rangle$ stabilizes a complex line $L$ in ${\rm
    H}_\C^n$. Moreover in that case the corresponding fixed point of
  $[A,B]$ is on the boundary $\partial {\rm H}_\C^n$, so that $\langle
  A, B \rangle$ is a $\C$-Fuchsian punctured torus group.
\end{prop} 

\Pf Indeed, with the notation of the proof of Theorem~\ref{decomp}, we
have the following variation of lemma~\ref{positive}:
 $$
{\bf X}(p_2,p_4,p_1,p_3) \in \R^- \ \iff \ \lambda_1 \in \R^-.
$$
Now by assumption $\lambda_1<0$, so that by Proposition~\ref{Xreal}
$p_1,p_2,p_3,p_4$ are on the boundary of a common complex line
$L$ and the pairs $(p_2,p_4)$ and $(p_1,p_3)$ separate each other.
But $A$ sends $(p_3,p_4)$ to $(p_2,p_1)$ and $B$ sends
$(p_2,p_3)$ to $(p_1,p_4)$, so $A$ and $B$ both stabilize $L$.
It is then a simple exercise in the hyperbolic plane to check that this 
combinatorics implies that $A$ and $B$ must both be loxodromic, with axes meeting inside the ideal  
quadrilateral $Q=(p_1,p_2,p_3,p_4)$. Therefore, $Q$ is disjoint from 
all its images by elements of the group generated by $A$ and $B$. This proves that 
$\la A, B \ra$ is discrete. Note that in general this quadrilateral is not a fundamental
domain for $\la A,B\ra$, as its images only tessellate the complex line $L$ if $[A,B]$ is 
parabolic. See Figure \ref{idealquad}.
\EPf

\begin{figure}
\centering
\scalebox{0.5}{\includegraphics{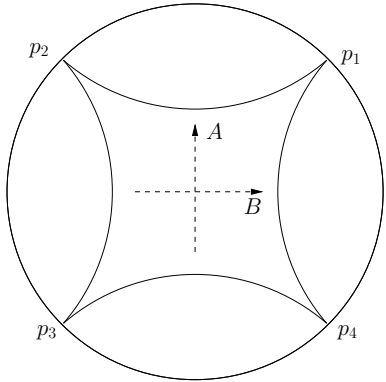}}
\caption{A fundamental fomain for the discrete punctured torus group in the case where $\lambda_1<0$\label{idealquad}}
\end{figure}

One can also interpret Proposition \ref{extremal} in terms of the Toledo invariant of the corresponding type-preserving 
representation of the fundamental group of the once-punctured torus. Given a representation $\rho$ of the 
fundamental group of a surface $\Sigma$ into ${\rm PU}(n,1)$, the \emph{Toledo invariant $\tau(\rho)$} is defined as the integral 
over $\Sigma$ of the pullback of the K\"ahler 2-form on $\HCn$ by an equivariant map $f :\tilde \Sigma \longrightarrow \HCn$:
\begin{equation}\label{toledo}
 \tau(\rho)=\int_\Sigma f^\star\omega.
\end{equation}
In the case where $\Sigma$ is non-compact, one should be careful that this integral is well-defined. This is guaranteed 
for instance by the condition that the map $f$ has finite energy (see \cite{KM}). The existence of a such a finite-energy 
equivariant map is guaranteed by the assumption that all peripheral loops on $\Sigma$ are mapped to parabolics by the 
representation (\cite{KM}). This assumption being made, it is possible to use an ideal triangulation of $\Sigma=\bigsqcup_i \Delta_i$
to compute the Toledo invariant, as in \cite{GP}. The result is that the Toledo invariant can be written as
\begin{equation}\label{splitol}
\tau(\rho) = \sum_i \int_{f(\Delta_i)} \omega,
\end{equation}
where each integral is computed over any 2-simplex with boundary $f(\partial \Delta_i)$.  But the integral of $\omega$ over 
any 2-simplex with boundary $\partial \Delta$ is equal to twice the Cartan invariant of $\Delta$ (this is Theorem 7.1.11 of \cite{G}). As a consequence, we see that for a type-preserving representation of the fundamental group of the once-punctured torus, 
the Toledo invariant is given by:
\begin{equation}\label{Cartantol}
\tau(\rho) = 2\left(\A(p_1,p_2,p_3) +\A(p_1,p_3,p_4)\right),
\end{equation}
with $p_1,...,p_4$ defined as previously. On the other hand, taking arguments in equation \eqref{crosstriple} gives:
\begin{equation}
 \arg\left(\X(p_2,p_4,p_1,p_3)\right)=\A(p_2,p_4,p_1)-\A(p_2,p_1,p_3)=\A(p_1,p_2,p_4)+\A(p_1,p_2,p_3)
\end{equation}
Therefore, if the cross-ratio $\X(p_2,p_4,p_1,p_3)$ is negative, then $\A(p_1,p_2,p_4)$ and $\A(p_1,p_2,p_3)$ 
must both be equal to either $\pi/2$ or $-\pi/2$. In view of \eqref{Cartantol}, this means that $|\tau(\rho)|=2\pi$.
But the Toledo invariant satisfies the Milnor-Wood inequality:
\begin{equation}\label{Milnor-Wood} |\tau(\rho)|\leq 2\pi(2g-2+p), \end{equation}
where equality holds if and only if the representation is discrete and preserves a complex line. In that case, the 
representation is called \textit{maximal} (see \cite{T,BIW}). In the case of the once punctured torus $g=p=1$, 
and therefore if $|\tau(\rho)|=2\pi$ the representation is maximal.

\section{Groups generated by real reflections}
We now use the criterion from Theorem~\ref{decomp2} to show that
various subgroups of ${\rm PU}(2,1)$ are generated by real
reflections. More accurately this means that they are the index 2
holomorphic subgroup of a group of isometries generated by real reflections.
\subsection{Mostow's lattices and other non-arithmetic lattices in ${\rm SU}(2,1)$}\label{mostow}
Mostow's lattices from \cite{M}  (revisited in \cite{DFP}) as well as the new non-arithmetic 
lattices in ${\rm SU}(2,1)$ studied by Deraux, Parker and the first author (see \cite{ParPau}, \cite{Pau}, 
\cite{DPP1} and \cite{DPP2}) are all symmetric complex reflection triangle groups. This means that they are 
generated by 3 complex reflections $R_1$, $R_2$ and $R_3$ which are in a symmetric configuration in the sense 
that there exists an isometry $J$ of order 3 such that $JR_iJ^{-1}=R_{i+1}$ (with $i$ mod 3). These groups 
are in fact contained (with index 1 or 3, depending on the parameters) in the group $\Gamma$ generated by $R_1$ and $J$.

It was shown in \cite{DFP} that Mostow's lattices are generated by real reflections, and in \cite{DPP2} this is extended 
to all symmetric complex reflection triangle groups. In both cases though, one finds an explicit real reflection which 
decomposes both holomorphic generators $R_1$ and $J$, which requires knowing explicit geometric properties of the group. 
Now the existence of such a real reflection follows immediately from the following consequence of Theorem~\ref{decomp}:

\begin{prop}\label{Creflpairs} If $R \in {\rm PU}(2,1)$ is a complex reflection or a complex reflection in a point, then 
for any $A \in {\rm PU}(2,1)$, the pair $(R,A)$ is $\R$-decomposable.
\end{prop}

 \Pf First assume that $R$ is a complex reflection about a  complex line $L$. Then $[R,A] = R R'$, where $R'=AR^{-1}A^{-1}$ 
is a complex reflection conjugate to $R^{-1}$. We denote by $L'$ its fixed complex line. The extensions of $L$ and $L'$ to 
$\C P^2$ intersect at a unique point $p\in \C P^2$, unless $L=L'$, in which case $[R,A]={\rm Id}$ and the result follows from
 Theorem \ref{decomp2}.  In general, $p$ is fixed by both $R$ and $R'$, thus by $[R,A]$; we distinguish 2 cases, depending on 
whether $p$ is in $\overline\HCd$ or outside of $\overline \HCd$.
\begin{enumerate}
 \item  First assume that $p\in\overline\HCd$. Consider lifts of $R$ and $R'$ to SU(2,1) such that the lift of $R$ has eigenvalues 
$e^{2i\phi},e^{-i\phi}, e^{-i\phi}$ (where the rotation angle of $R$ is $3\phi$), with $p$ corresponding to a $e^{-i\phi}$-eigenvector.
Likewise, the lift of $R'$ has eigenvalues $e^{-2i\phi},e^{i\phi}, e^{i\phi}$, with $p$ corresponding to a $e^{i\phi}$-eigenvector of 
$R'$. Then $p$ is a fixed point of $[R,A]=RR'$ in $\overline{\HCd}$ with corresponding eigenvalue 1. The result follows from Theorem 
\ref{decomp}.
\item If $p$ is outside of $\overline \HCd$, then $L$ and $L'$ are ultraparallel and $p$ is polar to their common perpendicular line, 
which we denote by $\tilde L$. The isometries $R$ and $R'$ act on $\tilde{L}$ by rotation through angles $3\phi$ and $-3\phi$ 
respectively. It is an elementary fact from plane hyperbolic geometry that the product of two elliptic elements with opposite 
rotation angles and distinct fixed points must be hyperbolic (it follows for instance from Lemma 7.38.2 of \cite{Be}). 
Therefore $RR'$ is loxodromic; moreover its eigenvalue of positive type is 1 (it is the product of those of $R$ and $R'$ which are 
inverse of one another). Therefore its eigenvalue spectrum must be either $\lbrace r, 1, 1/r\rbrace$ or $\lbrace -r, 1, -1/r\rbrace$ 
for some $r>0$. Geometrically, this means that the rotation angle of the loxodromic isometry $[R,A]$ is $0$ or $\pi$. By continuity, 
the rotation angle of $[R,A]$ must be the same for any value $\phi$ and any relative position of $L$ and $L'$ (with the restriction 
that they are ultraparallel). It is $0$ when $L=L'$ (as $R$ and $A$ are inverse of one another in this case). This means that 
$[R,A]$ has only positive eigenvalues, and the result follows from Theorem \ref{decomp}.
\end{enumerate}

Now assume that $R$ is a complex reflection about a point. As above, we write $[R,A]=RR'$ with $R'=AR^{-1}A^{-1}$. 
The eigenvalues of $R$ (respectively $R'$) are $\lbrace e^{i\phi},e^{i\phi},e^{-2i\phi}\rbrace$ 
(resp. $\lbrace e^{-i\phi},e^{-i\phi},e^{2i\phi}\rbrace$),  with $e^{2i\phi}$ (resp. $e^{-2i\phi}$) of negative type. 
This means that $R$ (resp. $R'$) acts on any complex line through its fixed point as a rotation of angle $3\phi$ (resp. 
$-3\phi$). Consider the complex line $\tilde L$ spanned by the fixed points of $R$ and $R'$. The action of $R$ and $R'$ on 
$\tilde L$ is the same as in the second item above, and this leads to the same conclusion.\EPf

\subsection{Groups with more than two generators}
Applying the criterion from Theorem~\ref{decomp} to a 2-generator subgroup of ${\rm PU}(2,1)$ is completely straightforward. 
For subgroups generated by more elements (e.g. Picard modular groups) one needs to be a bit more careful, as being generated 
by real reflections is stronger than having generators which are pairwise $\R$-decomposable. The following observation gives a 
way to bridge this gap, however its hypotheses are in general too restrictive and in practice we will need some more work to 
show that a given group is generated by real reflections.

\begin{lem}\label{genR} Let $\Gamma$ be a subgroup of ${\rm PU}(2,1)$ generated by
  $A_1,...,A_k$. If there exist real reflections
  $\sigma_1,...,\sigma_{k+1}$ such that 
  \begin{itemize}
  \item[(a)] $A_i=\sigma_i\sigma_{i+1}$ for
  $1 \leqslant i \leqslant k$, or
  \item[(b)] $A_i=\sigma_1\sigma_{i+1}$ for
  $1 \leqslant i \leqslant k$
 \end{itemize}
  then $\Gamma$ has index 2 in
  $\hat{\Gamma}=\langle \sigma_1,...,\sigma_{k+1} \rangle$. In
  particular such a $\Gamma$ is generated by real reflections.
\end{lem}
\Pf In each case, pairwise products of the $\sigma_i$ are in $\Gamma$, therefore $\Gamma$ is the index 2 holomorphic 
subgroup of $\hat{\Gamma}$. \EPf

\subsection{Picard modular groups}
We denote $\Gamma_d={\rm SU}(2,1,\mathcal{O}_d)$ the subgroup of ${\rm SU}(2,1)$ consisting of matrices with entries in 
$\mathcal{O}_d$, where $d$ is a positive squarefree integer and $\mathcal{O}_d$ denotes the ring of
integers of $\Q [\sqrt{-d}]$. Recall that $\mathcal{O}_d=\Z[i\sqrt{d}]$ when $d \equiv 1,2\mbox{ (mod  $4$)}$ and 
$\mathcal{O}_d=\Z[\frac{1+i\sqrt{d}}{2}]$ when $d \equiv 3\mbox{ (mod  $4$)}$.
The groups $\Gamma_d$ are usually called \emph{Picard modular groups}. Denote by $\Gamma_\infty(d)$ the 
stabilizer of $q_\infty$ in $\Gamma_d$ and $I_0 \in \Gamma_d$ given by:
\begin{equation}\label{I0}
I_0=
\begin{bmatrix}
0 & 0 & 1 \\
0 & -1 & 0 \\
1 & 0 & 0
\end{bmatrix}.
\end{equation}
The isometry $I_0$ is a complex reflection of order 2 about the complex line polar to the vector 
$v_0=\begin{bmatrix} 1 & 0 & 1\end{bmatrix}^T$. 
Elements of $\Gamma_\infty(d)$ are upper triangular matrices with units of $\mathcal{O}_d$ on the diagonal. In particular, 
$\Gamma_\infty(d)$ contains no Heisenberg dilations (see \eqref{stabinf}). In other words 
$\Gamma_\infty(d) < {\rm Isom}(\mathfrak{N})$ and therefore it fits into the exact sequence 
\eqref{exact} as:
$$
1 \longrightarrow \Gamma_\infty^v(d) \longrightarrow \Gamma_\infty(d) \overset{\Pi_*}{\longrightarrow} \Gamma_\infty^h(d) \rightarrow 1,
$$
where $\Gamma_\infty^v(d) < {\rm Isom}^+(\R)$ and $\Gamma_\infty^h(d) < {\rm Isom}^+(\C)$. A complete description of 
the subgroups $\Gamma_\infty^h(d)$ and $\Gamma_\infty^v(d)$ was given in \cite{FFPn} for $d \neq 1,3$ 
(see \cite{FFP} for $d=1$ and \cite{FalPar} for $d=3$).
\begin{lem}[\cite{FFPn}] \begin{itemize}
\item[(a)] $\Gamma_\infty^v(d) = \langle T_0 \rangle$, with $T_0=T_{[0,\sqrt{d}]}$ as in \eqref{stabinf}.
\item [(b)] $\Gamma_\infty^h(d) = {\rm Isom}^+(\mathcal{O}_d)$ if $d \equiv 3 \ (mod \, 4)$ ($d\neq 3$);   $\Gamma_\infty^h(d)$ has 
index 2 in ${\rm Isom}^+(\mathcal{O}_d)$ if $d \equiv 1,2 \ (mod \, 4)$ ($d\neq 1$).
\end{itemize}
\end{lem}
\Pf
We sketch the proof given in \cite{FFPn}. First, note that an element $P_{(z,t,\theta)}$ of 
${\rm Isom}(\mathfrak{N})$ is in the kernel of $\Pi_*$ if and only if $z=\theta=0$, see \eqref{exact}. This gives part $(a)$. Next, when 
$d\neq 1,3$, the only units of $\mathcal{O}_d$ are $\pm 1$ and thus the rotational part of an element 
$P_{(z,t,\theta)}\in\Gamma_\infty(d)$ must satisfy $\theta=0$ or $\theta=\pi$. In particular the subgroup of $\Gamma_\infty(d)$ formed 
by Heisenberg translations has index two. The same observation holds for the translation subgroup Isom$^+(\mathcal{O}_d)$, which is an 
extension of its translation subgroup $\Delta_d^t$ by the central involution $w\longmapsto -w$. 
The strategy used in \cite{FFPn} is to try to lift the generators of $\Delta_d^t$ to $\Gamma_\infty(d)$. The existence of such a lift can be verified using the explicit form of the map $\Pi_*$ given \eqref{exact}.
\begin{itemize}
\item[(I)] Assume that $d \equiv 3 \mbox{ (mod 4)}$ (and $d \neq 3$) . Then $\Delta^t=\langle \hat{T_1}, \hat{T_2} \rangle$ with: 
$\hat{T_1}=\left(\begin{matrix} 1 & 1 \\ 0 & 1\end{matrix}\right)$,
 $\hat{T_2}=\left(\begin{matrix} 1 & (1+i\sqrt{d})/2 \\ 0 & 1\end{matrix}\right)$.
Both $\hat{T_1}$ and $\hat{T_2}$ lift to $\Gamma_\infty(d)$, and their lifts $T_1$ and $T_2$ satisfy 
$[T_2,T_1]=T_0$ which generates $\Gamma_\infty^v(d)$. Therefore, $ \langle T_1, T_2 \rangle $ 
has index 2 in $\Gamma_\infty(d)$, and $\Gamma_\infty^h(d)=\la \hat T_1,\hat T_2\ra={\rm Isom}^+(\mathcal{O}_d)$. 

\item[(II)] Assume that $d \equiv 2 \mbox{ (mod 4)}$. Then $\Delta^t=\langle \hat{T_1}, \hat{T_2} \rangle$ with: 
$\hat{T_1}=\left(\begin{matrix} 1 & 1 \\ 0 & 1\end{matrix}\right)$, 
$\hat{T_2}=\left(\begin{matrix} 1 & i\sqrt{d} \\ 0 & 1\end{matrix}\right)$.
Now $\hat{T_2}$ lifts to $\Gamma_\infty(d)$, $\hat{T_1}$ does not but $\hat{T_1}^2$ does. 
This implies that $\Gamma_\infty^h(d)=\la \hat{T_1}^2,\hat{T_2}\ra$. The lifts $T_2$ and $T'_1$ satisfy $[T_2,T'_1]=T_0^4$ 
which generates an index 4 subgroup of $\Gamma_\infty^v(d)$. Therefore, $ \langle T'_1, T_2 \rangle $ has 
index 8 in $\Gamma_\infty(d)$.

\item[(III)] Assume that $d \equiv 1 \mbox{ (mod 4)}$ (and $d \neq 1$). Then $\Delta^t=\langle \hat{T_1}, \hat{T_2} \rangle$ 
with: $\hat{T_1}=\left(\begin{matrix} 1 & 1 \\ 0 & 1\end{matrix}\right)$, 
$\hat{T_2}=\left(\begin{matrix} 1 & i\sqrt{d} \\ 0 & 1\end{matrix}\right)$.
Now $\hat{T_1}$ and $\hat{T_2}$ do not lift to $\Gamma_\infty(d)$,  but $\hat{T_1}^2$ and $\hat{T_2}^2$ do. 
This implies that $\Gamma_\infty^h(d)=\la \hat{T_1}^2,\hat{T_2}^2\ra$, which has index 4 in $\la \hat{T_1},\hat{T_2}\ra$.
The lifts $T'_1$ and $T'_2$ satisfy $[T'_2,T'_1]=T_0^4$ which generates an index 4 subgroup of $\Gamma_\infty^v(d)$. 
Therefore, $ \langle T'_1, T'_2 \rangle $ has index 16 in $\Gamma_\infty(d)$.
\end{itemize}
\EPf

\begin{rmk}
 In cases (II) and (III) above, the translation directions of $T_1$ and $T_2$ are orthogonal, and so are the translation 
directions of the generating pairs of $\Gamma_\infty^h(d)$. In case (I) this is not true. However, the two elements $\hat{T_1}$ 
and $\hat{T_2}^2\hat{T_1}^{-1}$ have this property and generate an index two subgroup of $\Gamma_\infty^h(d)$. Therefore in 
view of the previous discussion if we make the following changes:
\begin{itemize}
 \item in case (I), $(\hat{T_1},\hat{T_2})\longmapsto(\hat{T_1},\hat{T_2}^2\hat{T_1}^{-1})$,
 \item in case (II), $(\hat{T_1},\hat{T_2})\longmapsto(\hat{T_1}^2,\hat{T_2})$,
 \item in case (III), $(\hat{T_1},\hat{T_2})\longmapsto(\hat{T_1}^2,\hat{T_2}^2)$,
\end{itemize}
we obtain finite index subgroups of $\Gamma_\infty^h(d)$ generated by pairs of translations that lift to $\Gamma_\infty(d)$
and have orthogonal directions. We will denote the lifts of the new ${\hat T}_1$ and ${\hat T}_2$ by $T_1$ and $T_2$, and from now on
 we will only consider the group $\la T_1,T_2\ra$.
\end{rmk}
 The central involution $w\longmapsto -w$ lifts to $\Gamma_\infty(d)$ as the transformation given by
\begin{equation}\label{R1}
 R_1=\begin{bmatrix}
      -1 & 0 & 0\\
       0 & 1 & 0\\
       0 & 0 & -1 
     \end{bmatrix}.
\end{equation}
In particular, $\Gamma_\infty(d)$ is generated by $R_1$ and its Heisenberg translation subgroup. Note that $R_1$ is a complex reflection
of order 2 about the complex line polar to the vector $v_1=\begin{bmatrix} 0 & 1 & 0\end{bmatrix}^T$.
\begin{lem}
 \begin{enumerate}
  \item The real reflection $\sigma_0$ about the standard real plane $\R^2\cap\HCd$ decomposes $R_1$, $I_0$ and $T_2$.
  \item The real reflection $R_1\sigma_0$ decomposes $T_1$.
\end{enumerate}

\end{lem}
\Pf
\begin{enumerate}
\item Clearly, $\sigma_0$ fixes the polar vector to the mirrors of $R_1$ and $I_0$. Therefore it decomposes these two transformations. 
Indeed the two transformations $I_0\sigma_0$  and $R_1\sigma_0$ fix the vectors $v_0$ and $v_1$ respectively and thus preserve the 
corresponding complex lines. They are therefore real reflections, by Lemmas \ref{refl1} and \ref{refl2}.  On the other hand, $T_2$ is a lift to $\Gamma_\infty(d)$ of a translation by a multiple of $i\sqrt{d}$. This implies that
$T_2$ is a Heisenberg translation by $[ik\sqrt{d},t_2]$, where $k$ is an integer and $t_2$ is chosen so that $T_2\in\Gamma_\infty(d)$.
In particular, this implies that the invariant fan of $T_2$ intersects the boundary of $\HCd$ along a vertical plane of the type 
$\lbrace x=x_2\rbrace$, for some $x_2\in\R$. The fixed $\R$-circle of $\sigma_0$ is the $x$-axis of the Heisenberg group, and is 
therefore orthogonal to this vertical plane. This implies that $\sigma_0$ decomposes $T_2$ by Proposition~\ref{decompCrefllox}.

\item A direct computation shows that in horospherical coordinates $R_1\sigma_0$ acts as $(z,t,u)\longmapsto (-\bar z,-t,u)$. Its 
restriction to the boundary is thus the reflection about the infinite $\R$-circle $\{(iy,0),y\in \R\}$. As $T_1$ is the Heisenberg 
translation by some $[k,it_1]$, where $k\in\Z$, its invariant fan intersects the boundary along a vertical plane $\{y=y_1\}$ for some 
$y_1\in\R$, and is therefore orthogonal to the fixed $\R$-circle of $R_1\sigma_0$. Again, by Proposition~\ref{decompCrefllox}, this implies that $R_1\sigma_0$ decomposes $T_1$ .

\end{enumerate}
\EPf

We will denote the corresponding real reflections as follows:
\begin{equation}
 \sigma_1=R_1\sigma_0, \, \sigma_2=T_2\sigma_0, \, \sigma_4=I_0\sigma_0\mbox{ and } \sigma_3=T_1\sigma_1=T_1R_1\sigma_0.
\end{equation}
\begin{prop}
For all $d \geqslant 1$, the subgroup $\langle \Gamma_\infty (d), I_0 \rangle$ of $\Gamma_d$ is generated by real reflections up to finite index.
\end{prop}
\Pf
Consider the group $\Gamma_\sigma=\la \sigma_j,j=0\cdots 4\ra$. By construction $\Gamma_\sigma$ contains  the group
$\langle  T_1,T_2,R_1,I_0\rangle$ which has finite index in $\langle \Gamma_\infty(d), I_0 \rangle$. 
On the other hand, we claim that all pairwise products of the $\sigma_i$ are in $\langle \Gamma_\infty(d), I_0 \rangle$. 
Indeed, each of the $\sigma_i$ is of the form $A_i\sigma_0$ with $A_i \in \langle \Gamma_\infty(d), I_0 \rangle$, therefore: 
$\sigma_i\sigma_j=A_i\sigma_0A_j\sigma_0=A_i\overline{A_j}$. The result follows by noting that $\Gamma_\infty(d)$ and $ I_0$ 
are preserved under complex conjugation of entries.
\EPf

Note that from the previous discussion, denoting by $T_1,T_2$ the lifts to $\Gamma_\infty(d)$ of generators of the orthogonal 
reflection subgroup of ${\rm Isom}^+(\mathcal{O}_d)$ (or their square when applicable), the index of 
$\langle  T_1,T_2,R_1,I_0\rangle$ in $\langle \Gamma_\infty(d), I_0 \rangle$ is 2 when 
$d \equiv 3 \mbox{ (mod 4)}$ (and $d \neq 3$),  4 when $d \equiv 2 \mbox{ (mod 4)}$, and 8 when  $d \equiv 1 \mbox{ (mod 4)}$ (and $d \neq 1$). For $d=3$ the 
index is 6 and for $d=1$ it is 16 (as there is an additional rotational part of order 3, 2 respectively). However in the 
latter cases we can recover the rotational part in the real reflection group by replacing $R_1$ in the definition of 
$\sigma_1$ by the appropriate complex reflection of order 4 or 6. 

For $d=1,2,3,7,11$ it is known from \cite{FalPar}, \cite{FFP} and $\cite{Z}$ that $\langle \Gamma_\infty(d), I_0 \rangle= \Gamma_d$, 
giving the following:
\begin{cor}\label{picardrefl}  For $d=3,7,11$ (respectively 2 and 1), $\Gamma_d={\rm PU}(2,1,\mathcal{O}_d)$ has a subgroup of 
index 2 (respectively 4 and 8) which is generated by real reflections.
\end{cor}
Note that for $d=3$, the full group $\Gamma_3$ is in fact generated by real reflections by Proposition~\ref{Creflpairs}, as it was 
shown in \cite{FalPar} (Proposition 5.11) that $\Gamma_3 = \langle J, R_1 \rangle$ with $J$ regular elliptic of order 3 and $R_1$ a complex reflection of order 6. 
 
 \vspace{.3cm}
{\bf Question:} Are all Picard modular groups ${\rm PU}(2,1,\mathcal{O}_d)$ generated by real
  reflections up to finite index?

\raggedright


\frenchspacing
\begin{thebibliography}{ZZ99}

\bibitem[Ag]{Ag} I. Agol; {\sl Finiteness of arithmetic Kleinian reflection groups.} International Congress of Mathematicians Vol. II, pp. 951--960, Eur. Math. Soc. , Z\"urich, 2006.




\bibitem[BM]{BM} A. Basmajian, B. Maskit; {\sl Space form isometries as commutators and products of involutions}. Trans. Amer. Math. Soc. {\bf 364} (2012), no. 9, 5015--5033.

\bibitem[Be]{Be}A.~Beardon; The geometry of discrete groups. Graduate Texts in Mathematics,  91. Springer-Verlag, New York, 1983.

\bibitem[BeMc]{BeMc} M. Belolipetsky, J. McLeod; {\sl Reflective and quasi-reflective Bianchi groups}. Preprint 2012, {\tt arXiv 1210.2759}. To appear in Transform. Groups.

\bibitem[Bi]{Bi} L. Bianchi; {\sl Sui gruppi di sostituzioni lineari con coefficienti appartenenti a corpi quadratici immaginari}. Math. Ann. {\bf 40} (1892), 332--412.



\bibitem[Br]{Br} U.~Brehm; {\sl The shape invariant of triangles and trigonometry in two-point homogeneous spaces}. Geom. Dedicata {\bf 33} (1990), 59--76.

\bibitem[BILW]{BILW} M.~Burger, A. ~Iozzi, F. ~Labourie, A. ~Wienhard; {\sl Maximal representations of surface groups : Symplectic Anosov structures }.
Pure and Applied Mathematics Quaterly  {\bf 1} (2005), no 2, 555--601.

\bibitem[BIW]{BIW} M.~Burger, A. ~Iozzi,  A. ~Wienhard; {\sl Higher Teichm\"uller spaces : from SL(2,$\R$) to other Lie groups}.
to appear in Handbook of Teichm\"uller theory.

\bibitem[Bow]{Bow}B.~Bowditch; {\sl Discrete parabolic groups}. J. Differential. Geometry. {\bf 38}, 1993, 559--583.

\bibitem[C]{C} E.~Cartan; {\sl Sur le groupe de la g\'eom\'etrie hypersph\'erique}. Comment. Math. Helv. {\bf 4} (1932), 158--171.

\bibitem[ChGr]{ChGr} S.~Chen, L.~Greenberg; {\sl Hyperbolic spaces}, in Contributions to Analysis. Academic Press, New York (1974), 
49--87.

\bibitem[Cho]{Cho} S.~Choi; {\sl Product of two involutions in complex and reversible hyperbolic geometry}, (2007), no. 1, Preprint 
(summary available in {\it Abstracts of the Korean Mathematics Society}).


\bibitem[DFP]{DFP} M. Deraux, E. Falbel, J. Paupert; {\sl New constructions of fundamental polyhedra in complex hyperbolic space}. Acta Math. {\bf 194} (2005), 155--201.

\bibitem[DPP1]{DPP1} M.~Deraux, J.R.~Parker, J.~Paupert; {\sl Census of the complex hyperbolic sporadic groups}. Experiment. Math. {\bf 20} (2011), 467--486.

\bibitem[DPP2]{DPP2} M.~Deraux, J.R.~Parker, J.~Paupert; {\sl A family of new non-arithmetic complex hyperbolic lattices}. In preparation.

\bibitem[F]{F} E.~Falbel; {\sl Geometric structures associated to triangulations as fixed point sets of involutions}. Topology Appl. {\bf 154} (2007), no. 6, 1041--1052.

\bibitem[FFP]{FFP} E. Falbel, G. Francsics, J. Parker; {\sl The geometry of the Gauss-Picard modular group}. Math. Ann. {\bf 349} (2011), no. 2, 459--508. 

\bibitem[FFPn]{FFPn} E. Falbel, G. Francsics, J. Parker; Unpublished notes.

\bibitem[FalPar]{FalPar} E. Falbel, J. Parker; {\sl The geometry of the Eisenstein-Picard modular group}. Duke Math. J. {\bf 131} (2006) 249--289.

\bibitem[FalPar2]{FalPar2} E. Falbel, J. Parker; {\sl The moduli space of the modular group in complex hyperbolic geometry}. Invent. Math. {\bf 152} (2003), no. 1, 57--88 .

\bibitem[FalPau]{FalPau} E.~Falbel, J.~Paupert;
{\sl Fundamental domains for finite subgroups in $U(2)$ and
  configurations of Lagrangians}. Geom. Dedicata {\bf 109} (2004),
221--238. 

\bibitem[FZ]{FZ} E. Falbel, V. Zocca; {\sl A Poincar{\'e}'s polyhedron
    theorem for complex hyperbolic geometry}. J. Reine
    Angew. Math. {\bf 516} (1999), 138--158.

\bibitem[G]{G} W.M.~Goldman; Complex Hyperbolic Geometry.
Oxford Mathematical Monographs. Oxford University Press (1999).

\bibitem[GoP]{GoP} W.M.~Goldman, J.~Parker, {\sl Dirichlet Polyhedra for Dihedral Groups Acting on Complex Hyperbolic Space}.  Jour. Geom. Analysis {\bf 91} 1992 (6) 517--554.

\bibitem[GP]{GP} N.~Gusevskii, J.~Parker, {\sl Complex hyperbolic quasi-Fuchsian groups and Toledo's invariant}. 
Special volume dedicated to the memory of Hanna Miriam Sandler (1960–1999). Geom. Ded. {\bf 97} (2003) 151--185.

\bibitem[GonP]{GonP} K.~Gongopadhyay, J.~Parker, {\sl Reversible complex hyperbolic isometries}. 
 Lin. Alg. Appl. {\bf 438} (6) (2013) 2728--2739.

\bibitem[H]{H} R.-P. Holzapfel; {\sl Invariants of arithmetic ball quotient surfaces.} Math. Nachr. {\bf 103} (1981), 117--153.

\bibitem[KR]{KR} A. Koranyi, H.M. Reimann; {\sl The complex cross-ratio on the Heisenberg group.} Enseign. Math. {\bf 33} (1987), 291--300.

\bibitem[KM]{KM} A. Koziarz, J. Maubon; {\sl Harmonic maps and representations of non-uniform lattices of PU(m,1)}, 
Ann. Inst. Fourier (Grenoble), {\bf 58}, 2008 (2), 507–-558.


\bibitem[M]{M} G.D.~Mostow; {\sl On a remarkable class of polyhedra 
in complex hyperbolic space}. Pacific J.\ Maths.\ {\bf 86} (1980),
171--276.


\bibitem[ParPau]{ParPau} J.R.~Parker, J.~Paupert; {\sl Unfaithful complex hyperbolic
triangle groups II: Higher order reflections}. Pacific J.\ Maths.\ {\bf 239} (2009), 357--389.

\bibitem[Pau]{Pau} J. Paupert; {\sl Unfaithful complex hyperbolic triangle groups III: arithmeticity and commensurability of sporadic groups}. Pacific J. Maths. {\bf 245} (2010), 359--372.

\bibitem[PW]{PW} J. Paupert, P. Will; {\sl Involutions and commutators in $\rm{PU}(2,1)$}. In preparation.

\bibitem[Pra]{Pra} G. Prasad; {\sl Volumes of S-arithmetic quotients of semi-simple groups}. Publ. Math. IHES {\bf 69} (1989), 91--117.

\bibitem[Sh]{Sh} M.K. Shaiheev; {\sl Reflective subgroups in Bianchi groups.} Selecta Math. Soviet {\bf 9} (1990), 315--322.

\bibitem[St]{St} M. Stover; {\sl Volumes of Picard modular surfaces}. Proc. Amer. Math. Soc. {\bf 139} (2011), 3045--3056.


\bibitem[T]{T} D. ~Toledo; {\sl Representations of surfaces groups  in complex hyperbolic space}. J. Differential Geom. {\bf 29} (1989), 125--133.

\bibitem[W1]{W1} P.~Will; Groupes libres, groupes triangulaires et tore \'epoint\'e dans ${\rm PU}(2,1)$. Ph.D. Thesis, Universit\'e Paris 6 (2006). Available at: {\tt http://www-fourier.ujf-grenoble.fr/$\sim$will/these.pdf}

\bibitem[W2]{W2} P.~Will; {\sl Traces, Cross-ratios and 2-generator Subgroups of ${\rm PU}(2,1)$}. Can. J. Math. {\bf 61} (2009), 1407--1436.

\bibitem[W3]{W3} P.~Will; {\sl The punctured torus and Lagrangian triangle groups in PU(2,1)}.  {\bf 602} (2007), 95--121.

\bibitem[Z]{Z} T. Zhao; {\sl Generators for the Euclidean Picard modular groups}. Trans. Amer. Math. Soc. {\bf 364} (2012), 3241--3263.

\end{thebibliography}
\end{document}